\documentclass[a4paper,11pt]{amsart}

\usepackage{etex}
\usepackage{amsmath,amssymb,amsthm,amsfonts}
\usepackage[dvipsnames]{xcolor}
\usepackage[all]{xy}
\SelectTips{cm}{}
\usepackage{ifthen}
\usepackage{graphicx}
\usepackage{psfrag}
\usepackage{enumitem}
\usepackage{booktabs}
\usepackage{setspace}
\usepackage{hyperref}
\usepackage{stmaryrd}
\hypersetup{
    colorlinks,
    linkcolor={red!50!black},
    citecolor={blue!50!black},
    urlcolor={blue!80!black}
}
\usepackage{mathtools}
\newcommand{\hyref}[2]{\hyperref[#2]{#1~\ref*{#2}}}

\usepackage{tikz}
\usetikzlibrary{matrix,arrows,decorations.pathmorphing,calc}

\newcommand{\lineinsertion}[3]{#1(#2#3)}

\newcommand{\tack}{\triangleright}

\newboolean{ebookreader} 
\setboolean{ebookreader}{false}
\ifthenelse{\boolean{ebookreader}}{
\usepackage[paperwidth=160mm,paperheight=120mm,left=5mm,right=5mm,top=5mm,bottom =5mm]{geometry}
\setlength{\pdfpagewidth}{\paperwidth}
\setlength{\pdfpageheight}{\paperheight}
\pagestyle{empty}
\hfuzz=5mm
\emergencystretch=15mm}{
}

\definecolor{prussian}{RGB}{0,49,83}
\definecolor{hunter}{RGB}{53,94,59}

\newcommand{\OLD}[1]{}

\newcommand{\uLambda}{{\underline{\Lambda}}}
\newcommand{\smuLambda}{{\smash{\text{\resizebox{!}{1.2ex}{$\underline{\Lambda}$}}}}}

\newcommand{\SOD}{{\rm (\ddag)}}
\newcommand{\SOT}{{\rm (\dag)}}

\newcommand{\fixedwidthtabular}{ \noindent \begin{tabular}{@{} p{0.08\textwidth} @{} p{0.92\textwidth} @{} } }

\newcommand{\assumption}[2]{ \medskip\noindent
                             \begin{tabular}{@{} p{0.077\textwidth} @{} p{0.92\textwidth} @{} }
                               #1 & #2 
                             \end{tabular}\medskip }

\newcommand{\bib}[6]{{\bibitem{#2} #3, {\emph{#4},} #5#6.}}
\newcommand{\arXiv}[1]{{\href{http://arxiv.org/abs/#1}{\texttt{arXiv:#1}}}}

\newcommand{\poset}{\mathcal P}
\newcommand{\posett}{\mathcal P\!}
\newcommand{\posetaut}{\mathcal{GP}}
\newcommand{\posetautt}{\mathcal{GP}\!}

\newcommand{\TTT}{\mathsf{T}\!}  

\newcommand{\SSS}{\mathsf{S}}    
\newcommand{\Naka}{\nu}          

\newcommand{\Db}{\mathcal{D}^b}  
\renewcommand{\AA}{\mathcal{A}}

\newcommand{\CC}{\mathcal{C}}
\newcommand{\DD}{\mathcal{D}}
\newcommand{\EE}{\mathcal{E}}

\newcommand{\II}{\mathcal{I}}
\newcommand{\JJ}{\mathcal{J}}

\newcommand{\OO}{\mathcal{O}}

\renewcommand{\SS}{\mathcal{S}}
\newcommand{\XX}{\mathcal{X}}
\newcommand{\YY}{\mathcal{Y}}
\newcommand{\ZZ}{\mathcal{Z}}
\newcommand{\UU}{\mathcal{U}}

\newcommand{\IP}{\mathbb{P}}

\newcommand{\IZ}{\mathbb{Z}}
\newcommand{\IN}{\mathbb{N}}

\newcommand{\kk}{\mathbf{k}}

\newcommand{\A}[1]{A_{#1}}

\newcommand{\categoryfont}{\mathsf}
\newcommand{\modules}[1]{#1\operatorname{\categoryfont{-mod}}}                
\newcommand{\stmodules}[1]{#1\operatorname{\categoryfont{-\underline{mod}}}}  
\newcommand{\projectives}[1]{#1\!\operatorname{\categoryfont{-proj}}}           
\DeclareMathOperator{\Modules}{\categoryfont{Mod}}
\DeclareMathOperator{\add}{\categoryfont{add}}
\DeclareMathOperator{\thick}{\categoryfont{thick}}
\newcommand{\ind}{\categoryfont{ind}}
\newcommand{\Perf}{\mathsf{Perf}}
\newcommand{\Kb}{\mathsf{K}^b}  

\newcommand{\recollto}{%
        \mathrel{\vcenter{\mathsurround0pt
                \ialign{##\crcr
                        \noalign{\nointerlineskip}$\,\,\scriptstyle\leftarrow$\crcr
                        \noalign{\nointerlineskip}$\longrightarrow$\crcr
                        \noalign{\nointerlineskip}$\,\,\scriptstyle\leftarrow$\crcr
                }%
        }}%
}

\newcommand{\weakrectoup}{%
        \mathrel{\vcenter{\mathsurround0pt
                \ialign{##\crcr
                        \noalign{\nointerlineskip}$\,\,\scriptstyle\leftarrow$\crcr
                        \noalign{\nointerlineskip}$\longrightarrow$\crcr
                        \noalign{\nointerlineskip}$\phantom{\scriptstyle\leftarrow}$\crcr
                }%
        }}%
}

\newcommand{\recoll}[3]{#1 \recollto #2 \recollto #3}                
\newcommand{\wrecoll}[3]{#1 \,\raisebox{-0.2ex}{$\displaystyle\weakrectoup$}\, #2 \,\raisebox{-0.2ex}{$\displaystyle\weakrectoup$}\, #3}       

\newcommand{\jincl}{\jmath}  
\newcommand{\iincl}{\iota}   

\newcommand{\sod}[1]{{\langle #1 \rangle}}
\newcommand{\genby}[1]{{\langle #1 \rangle}}

\newcommand{\orth}{^\perp}
\newcommand{\op}{^\text{op}}
\newcommand{\lorth}{{}^\perp}

\newcommand{\inv}{^{-1}}
\newcommand{\blank}{-}        
\newcommand{\hf}{{\text{hf}}} 

\newcommand{\coloneqq}{\mathrel{\mathop:}=}
\newcommand{\eqqcolon}{=\mathrel{\mathop:}}  

\DeclareMathOperator{\Coh}{Coh}

\DeclareMathOperator{\coker}{coker}

\DeclareMathOperator{\RHom}{RHom}
\DeclareMathOperator{\Hom}{Hom}
\DeclareMathOperator{\End}{End}
\DeclareMathOperator{\Aut}{Aut}
\DeclareMathOperator{\Ext}{Ext}

\DeclareMathOperator{\id}{id}

\DeclareMathOperator{\supp}{supp}

\DeclareMathOperator{\chara}{char}
\DeclareMathOperator{\CB}{CB}
\DeclareMathOperator{\CI}{CI}
\DeclareMathOperator{\projdim}{projdim}
\DeclareMathOperator{\injdim}{injdim}
\DeclareMathOperator{\gldim}{gl{.}dim}
\newcommand{\xyinjar}{\ar@{^{(}->}}
\newcommand{\xysurjar}{\ar@{->>}}
\newcommand{\arrd}{ \ar@{-}[r] \ar@{=}[d] }

\newcommand{\isom}{ \text{{\hspace{0.48em}\raisebox{0.8ex}{${\scriptscriptstyle\sim}$}}}
                    \hspace{-0.65em}{\rightarrow}\hspace{0.3em}} 
\newcommand{\into}{\hookrightarrow}
\newcommand{\onto}{\twoheadrightarrow}

\newtheorem{theorem}{Theorem}[section]

\newtheorem*{theorem*}{Theorem}
\newtheorem{proposition}[theorem]{Proposition}
\newtheorem*{proposition*}{Proposition}
\newtheorem{lemma}[theorem]{Lemma}
\newtheorem*{lemma*}{Lemma}
\newtheorem*{conjecture}{Conjecture}
\newtheorem{corollary}[theorem]{Corollary}

\theoremstyle{definition}
\newtheorem{definition}[theorem]{Definition}
\newtheorem{remark}[theorem]{Remark}

\newtheorem*{remark*}{Remark}

\newtheorem*{question*}{Question}

\makeatletter
\newtheoremstyle{bolded}
   {8.0pt plus 2.0pt minus 4.0pt}{8.0pt plus 2.0pt minus 4.0pt}
   {}{}  
   {\bfseries}{}%
   { }    
   {\thmname{#1}\thmnumber{\@ifnotempty{#1}{ }\@upn{#2}}%
     \thmnote{ {\bfseries(#3)}}\bfseries{.}}%
\makeatother

\theoremstyle{bolded}
\newtheorem*{example*}{Example}
\newtheorem*{examples*}{Examples}
\newtheorem{example}[theorem]{Example}

\newcommand{\tikzcanonicalalgebra}{
\begin{tikzpicture}[yscale=0.6,xscale=1.0]

\node (source) at (0,0)     {0};
\node (sink)   at (6,0)     {1};
\node (a1-1)   at (1,1)     {$\bullet$};
\node (a1-2)   at (2.2,1.5)   {$\bullet$};
\node (a1-3)   at (3.8,1.5)   {$\bullet$};
\node (a1-4)   at (5,1)     {$\bullet$};
\node (a2-1)   at (1.5,0)   {$\bullet$};
\node (a2-2)   at (3,0)     {$\bullet$};
\node (a2-3)   at (4.5,0)   {$\bullet$};
\node (a3-1)   at (1,-1)  {$\bullet$};
\node (a3-2)   at (2,-1.35)  {$\bullet$};
\node (a3-3)   at (3,-1.5)  {$\bullet$};
\node (a3-4)   at (4,-1.35)  {$\bullet$};
\node (a3-5)   at (5,-1)  {$\bullet$};

\draw [->] (source) edge [bend left=10]   (a1-1);
\draw [->] (a1-1)   edge [bend left=10]   (a1-2);
\draw [->] (a1-2)   edge [bend left= 7]   (a1-3);
\draw [->] (a1-3)   edge [bend left=10]   (a1-4);
\draw [->] (a1-4)   edge [bend left=10]   (sink);

\draw [->] (source) edge (a2-1);
\draw [->] (a2-1)   edge (a2-2);
\draw [->] (a2-2)   edge (a2-3);
\draw [->] (a2-3)   edge (sink);

\draw [->] (source) edge [bend right=10]  (a3-1);
\draw [->] (a3-1)   edge [bend right= 5]  (a3-2);
\draw [->] (a3-2)   edge [bend right= 3]  (a3-3);
\draw [->] (a3-3)   edge [bend right= 3]  (a3-4);
\draw [->] (a3-4)   edge [bend right= 5]  (a3-5);
\draw [->] (a3-5)   edge [bend right=10]  (sink);
\end{tikzpicture}
}
\newcommand{\tikzcircular}{
\begin{tikzpicture}[scale=0.5]
\tikzstyle{every node}=[font=\small]

\foreach \x in {1,...,7} {
 \node[shape=circle,draw,inner sep=1pt] at ({160-(\x-1)*360/7}:3) {\x};
 \draw [->] ({160-(\x-1)*360/7-10}:3) arc ({160-(\x-1)*360/7-10}:{160-\x*360/7+10}:3) node at ({160-(\x-0.3)*360/7+10}:3+0.4)  {$a_\x$};
}

\end{tikzpicture}
\hspace{1em}
\begin{tikzpicture}[scale=0.5]
\tikzstyle{every node}=[font=\small]

\foreach \x in {1,...,7} {
 \node[shape=circle,draw,inner sep=1pt] at ({160-(\x-1)*360/7}:3) {\x};
 \draw [->] ({160-(\x-1)*360/7-10}:3) arc ({160-(\x-1)*360/7-10}:{160-\x*360/7+10}:3);
}

\foreach \x in {1,...,6} {
 \draw [densely dotted] ({160-(\x-0.5)*360/7-2}:3-0.2) [bend right] edge ({160-(\x+0.5)*360/7+2}:3-0.2);
}

\end{tikzpicture}
\hspace{1em}
\begin{tikzpicture}[scale=0.5]
\tikzstyle{every node}=[font=\small]

\foreach \x in {1,...,7} {
 \node[shape=circle,draw,inner sep=1pt] at ({160-(\x-1)*360/7}:3) {\x};
 \draw [->] ({160-(\x-1)*360/7-10}:3) arc ({160-(\x-1)*360/7-10}:{160-\x*360/7+10}:3);
}

\foreach \x in {5} {
 \draw [densely dotted] ({160-(\x-0.5)*360/7-2}:3-0.2) [bend right] edge ({160-(\x+1.5)*360/7+2}:3-0.2);
}

\end{tikzpicture}
}
\newcommand{\tikzdiscreteX}{
\begin{tikzpicture}[scale=0.8]

\def\b{24}
\def\h{5}
\def\tikzsize{\scriptstyle}

\clip (\h+0.7,-0.2) rectangle (\b-3.2,\h-0.7);

\foreach \x in {3,7,...,\b} {
 \filldraw[fill=gray!20,draw=gray!20,line width=15pt] (\x+1,1) -- (\x+\h,\h) -- (\x+\h+2,\h) -- (\x+2,0) -- (\x+1,1) -- (\x+\h,\h);
}

\foreach \x in {8,12,...,\b} {
 \filldraw[fill=gray!20,draw=gray!20,line width=15pt] (\x-1,1) -- (\x-\h,\h) -- (\x-\h-2,\h) -- (\x-2,0) -- (\x-1,1) -- (\x-\h,\h);
}

\foreach \x in {11} {
 \filldraw[fill=gray!30,draw=gray!30,line width=15pt] (\x+1,1) -- (\x+\h,\h) -- (\x+\h+2,\h) -- (\x+2,0) -- (\x+1,1) -- (\x+\h,\h);
}

\foreach \x in {12} {
 \filldraw[fill=gray!30,draw=gray!30,line width=15pt] (\x-1,1) -- (\x-\h,\h) -- (\x-\h-2,\h) -- (\x-2,0) -- (\x-1,1) -- (\x-\h,\h);
}

\foreach \x in {0,...,\b} {
 \foreach \y in {0,0.5,...,\h} {
\filldraw[fill=gray!50, draw=gray!50] (\x+\y,\y) circle (2pt);
}}

\foreach \x in {7,11,...,\b} {
 \filldraw[fill=black,draw=black] (\x,0)       circle (2pt);
 \filldraw[fill=black,draw=black] (\x+1,0)     circle (2pt);
 \filldraw[fill=black,draw=black] (\x+0.5,0.5) circle (2pt);
 \filldraw[fill=black,draw=black] (\x+0.5,1.5) circle (2pt);
 \filldraw[fill=black,draw=black] (\x+2.5,3.5) circle (2pt);
}

\filldraw[fill=black,draw=black]                         (11,1)     circle      (2pt)     node[left]  {$\tikzsize E$};
\filldraw[fill=black,draw=black]                         (12,1)     circle      (2pt)     node[right] {$\tikzsize \tau\inv E$};
\filldraw[fill=black,draw=black,yshift=-2pt,xshift=-2pt] (11.5,1.5) rectangle ++(4pt,4pt) node[above] {$\tikzsize X$};
\filldraw[fill=black,draw=black,yshift=-2pt,xshift=-2pt] (15.5,1.5) rectangle ++(4pt,4pt) node[above] {$\tikzsize X[2]$};
\filldraw[fill=black,draw=black,yshift=-2pt,xshift=-2pt] ( 7.5,1.5) rectangle ++(4pt,4pt);
\filldraw[fill=black,draw=black,yshift=-2pt,xshift=-2pt] (19.5,1.5) rectangle ++(4pt,4pt);

\end{tikzpicture}
}
\newcommand{\tikzdiscreteZ}{
\begin{tikzpicture}[scale=0.8]

\def\b{20}
\def\h{5}

\clip (\h+1.2,-0.2) rectangle (\b-1.2,\h+0.2);

\foreach \x in {1,5,...,\b} {
 \filldraw[fill=gray!20,draw=gray!20,line width=15pt] (\x,0) -- (\x+\h,\h) -- (\x+\h+2,\h) -- (\x+2,0) -- (\x,0) -- (\x+\h,\h);
}

\foreach \x in {0,...,\b} {
 \foreach \y in {0,0.5,...,\h} {
\filldraw[fill=gray!50,draw=gray!50] (\x+\y,\y) circle (2pt);
}}

\foreach \x in {0,4,...,\b} {
 \foreach \y in {0,0.5,...,\h} {
\filldraw[fill=black,draw=black] (\x+\y,\y) circle (2pt);
}}

\end{tikzpicture}
}
\newcommand{\tikzexample}{
\begin{tikzpicture}[scale=0.5]
\tikzstyle{every node}  = [circle, draw, fill=black!80, inner sep=0pt, minimum width=4pt]

\pgfmathsetmacro{\dynkin} { 0.5}
\pgfmathsetmacro{\cycle}  {11}
\pgfmathsetmacro{\columns}{18}

\node  (extDynkin_top)  at  (\dynkin + 2.5, 1)  {} ;
\node  (extDynkin_1)    at  (\dynkin      , 0)    {} ;
\node  (extDynkin_2)    at  (\dynkin + 1  , 0)    {} ;
\node  (extDynkin_3)    at  (\dynkin + 2  , 0)    {} ;
\node  (extDynkin_4)    at  (\dynkin + 3  , 0)    {} ;
\node  (extDynkin_5)    at  (\dynkin + 4  , 0)    {} ;
\node  (extDynkin_6)    at  (\dynkin + 5  , 0)    {} ;

\draw  (extDynkin_top) -- (extDynkin_1) ;
\draw  (extDynkin_top) -- (extDynkin_2) ;
\draw  (extDynkin_top) -- (extDynkin_3) ;
\draw  (extDynkin_top) -- (extDynkin_4) ;
\draw  (extDynkin_top) -- (extDynkin_5) ;
\draw  (extDynkin_top) -- (extDynkin_6) ;

\node[draw=none,rectangle,fill=black!0] at (\dynkin + 2.5,-2) {$\poset(\kk\tilde D_4) = \poset(\Lambda(2,3,3))$} ;
\node[draw=none,rectangle,fill=black!0] at (\dynkin + 2.5,-3) {Examples \ref{ex:hereditary}, \ref{ex:DDA_spherical_subcategory} } ;

\node  (square_top)    at  (\cycle + 1, 1)  {} ;
\node  (square_left)   at  (\cycle    , 0)  {} ;
\node  (square_right)  at  (\cycle + 2, 0)  {} ;
\node  (square_bottom) at  (\cycle + 1,-1)  {} ;

\draw  (square_top) -- (square_left) -- (square_bottom) -- (square_right) -- (square_top) ;

\node[draw=none,rectangle,fill=black!0] at (\cycle + 1,-3) {Example \ref{ex:poset_cycle}} ;

\pgfmathsetmacro{\left}  {\columns}
\pgfmathsetmacro{\middle}{\columns + 2.5}	
\pgfmathsetmacro{\right} {\columns + 5}

\node  (height_a3)    at  (\left,   1)  {} ;
\node  (height_a2)    at  (\left,   0)  {} ;
\node  (height_a1)    at  (\left,  -1)  {} ;

\node  (height_b4)    at  (\middle, 2)  {} ;
\node  (height_b3)    at  (\middle, 1)  {} ;
\node  (height_b2)    at  (\middle, 0)  {} ;
\node  (height_b1)    at  (\middle,-1)  {} ;

\node  (height_c5)    at  (\right, 3)  {} ;
\node  (height_c4)    at  (\right, 2)  {} ;
\node  (height_c3)    at  (\right, 1)  {} ;
\node  (height_c2)    at  (\right, 0)  {} ;
\node  (height_c1)    at  (\right,-1)  {} ;

\draw (height_a1) -- (height_a2) -- (height_a3) ;
\draw (height_b1) -- (height_b2) -- (height_b3) -- (height_b4) ;
\draw (height_c1) -- (height_c2) -- (height_c3) -- (height_c4) -- (height_c5) ;

\node[draw=none,rectangle,fill=black!0] at (\left   + 0.2,-2) {$\poset(\Lambda_3)$} ;	
\node[draw=none,rectangle,fill=black!0] at (\middle + 0.2,-2) {$\poset(\Lambda_4)$} ;	
\node[draw=none,rectangle,fill=black!0] at (\right  + 0.2,-2) {$\poset(\Lambda_5)$} ;	

\node[draw=none,rectangle,fill=black!0] at (\columns + 2.5,-3) {Corollary \ref{cor:increasing_heights}} ;

\end{tikzpicture}
}
\newcommand{\tikztubes}{
\begin{tikzpicture}[scale=0.8]
\node (ht1) at (0,0)   {};
\node (ht2) at (0,1)   {};
\node (ht3) at (0,2)   {};
\node (ht4) at (0,3)   {};
\node       at (0,3.3) {$\vdots$};

\filldraw[fill=black,draw=black] (ht1) [xshift=-2pt,yshift=-2pt] rectangle ++(4pt,4pt);
\filldraw[fill=black,draw=black] (ht2) circle (2pt);
\filldraw[fill=black,draw=black] (ht3) circle (2pt);

\draw [->] (ht1) edge [bend left] (ht2);
\draw [->] (ht2) edge [bend left] (ht1);
\draw [->] (ht2) edge [bend left] (ht3);
\draw [->] (ht3) edge [bend left] (ht2);
\draw [->] (ht3) edge [bend left] (ht4);
\draw [->] (ht4) edge [bend left] (ht3);

\foreach \x in {0,...,3} {
 \node (et\x-0) at ({2*\x+3},0) {};
 \filldraw[fill=black,draw=black] (et\x-0) circle (2pt);
 \node (et\x-2) at ({2*\x+3},2) {};
 \filldraw[fill=black,draw=black] (et\x-2) [xshift=-2pt,yshift=-2pt] rectangle ++(4pt,4pt);
}
\foreach \x in {1,...,3} {
 \node (et\x-1) at ({2*\x+2},1) {};
 \filldraw[fill=black,draw=black] (et\x-1) circle (2pt);
 \node (et\x-3) at ({2*\x+2},3) {};
}
\node         at (4,3.3) {$\vdots$};
\node         at (6,3.3) {$\vdots$};
\node         at (8,3.3) {$\vdots$};
\node (et0-4) at (3,3.4) {};
\node (et3-4) at (9,3.4) {};

\draw [->] (et0-0) edge (et1-1);
\draw [->] (et1-0) edge (et2-1);
\draw [->] (et2-0) edge (et3-1);
\draw [->] (et1-1) edge (et1-0);
\draw [->] (et2-1) edge (et2-0);
\draw [->] (et3-1) edge (et3-0);

\draw [->] (et0-2) edge (et1-1);
\draw [->] (et1-2) edge (et2-1);
\draw [->] (et2-2) edge (et3-1);
\draw [->] (et1-1) edge (et1-2);
\draw [->] (et2-1) edge (et2-2);
\draw [->] (et3-1) edge (et3-2);

\draw [->] (et0-2) edge (et1-3);
\draw [->] (et1-2) edge (et2-3);
\draw [->] (et2-2) edge (et3-3);
\draw [->] (et1-3) edge (et1-2);
\draw [->] (et2-3) edge (et2-2);
\draw [->] (et3-3) edge (et3-2);

\draw [dotted] (et0-0) edge (et0-4);
\draw [dotted] (et3-0) edge (et3-4);
\end{tikzpicture}
}

\newcommand{\Top}{2}
\newcommand{\Left}{3}
\newcommand{\Right}{1}
\newcommand{\Arup}{\alpha_2}
\newcommand{\Ardown}{\alpha_1}

\title[Spherical subcategories in representation theory]{Spherical subcategories\\in representation theory}

\author{Andreas Hochenegger}
\author{Martin Kalck}
\author{David Ploog}

\begin{document}

\begin{abstract}
We introduce a new invariant for triangulated categories: the poset of spherical subcategories ordered by inclusion.
This yields several numerical invariants, like the cardinality and the height of the poset.
We explicitly describe spherical subcategories and their poset structure for derived categories of certain finite-dimensional algebras. 
\end{abstract}

\begingroup\renewcommand\thefootnote{}%

\addtocounter{footnote}{-2}\endgroup
\maketitle

{\small \tableofcontents}

\thispagestyle{empty}
\vspace{-0.5cm}

\addtocontents{toc}{\protect\setcounter{tocdepth}{-1}}  
\section*{Introduction} 
\addtocontents{toc}{\protect\setcounter{tocdepth}{1}}   

\noindent
Our previous work \cite{HKP} associates a natural triangulated subcategory to an object of a $\kk$-linear triangulated category with two dimensional graded endomorphism algebra. Such an object is called \emph{spherelike} and the associated subcategory is called \emph{spherical subcategory}. It is the unique maximal subcategory in which the object becomes spherical in the sense of Seidel \& Thomas \cite{Seidel-Thomas}, who introduced them to construct symmetries predicted by Kontsevich's Homological Mirror Symmetry Conjecture.

In this paper, we extend this study in several ways.
Firstly, we observe that spherelike objects sometimes distinguish derived categories of algebras. For example, the derived categories of the following gentle 2-cycle algebras

\medskip\noindent
\resizebox{\textwidth}{!}{%
$\Lambda_1 = \kk \Big( \xymatrix@C=1.3em{
          1 \ar@<0.5ex>[r]^a \ar@<-0.5ex>[r]_b &
          2 \ar@<0.5ex>[r]^a \ar@<-0.5ex>[r]_b & 3 }\Big) / (a^2,b^2)
   , \qquad
\Lambda_2 = \kk \bigg( \raisebox{3ex}{\xymatrix@C=1.6em@R=2ex{
        & \Top \ar@/_0.5pc/[dl]_{\Ardown} \\
 \Left \ar@<-0.5ex>[rr]_\beta \ar@<0.5ex>[rr]^\gamma & & \Right \ar@/_0.5pc/[ul]_{\Arup} }} \bigg) / (\Arup \beta, \gamma \Ardown,\Ardown\Arup)
$}

\smallskip
\noindent
have different spherelike objects, hence are not equivalent; see \hyref{Subsection}{sec:derived-invariant}. All previously known derived invariants for these algebras coincide. The non-equivalence of the derived categories was already shown independently by Bobi\'nski in \cite{Bobinski} and the second author in \cite{Kalck}.
Moreover, examples in algebraic geometry suggest that these invariants carry geometric information: the behaviour under blow-ups is given in \cite[Prop.~5.2]{HKP}, and the spherical subcategories behave interestingly, see e.g\ \cite[Ex.~5.8]{HKP} for ruled surfaces over elliptic curves. Moreover, there is a rich theory of divisors $D$ in projective surfaces whose structure sheaf $\OO_D$ is spherelike; see \cite{Hochenegger-Ploog}.

Secondly, we observe that a spherical subcategory can be properly contained in another spherical subcategory. This yields a partial order on the set of spherical subcategories and indeed an invariant of triangulated categories. Several coarser invariants like the cardinality and the height of the poset are derived from this. We illustrate how these invariants capture the complexity of the category in several algebraic examples. 
 
Thirdly, we develop a general machinery to construct spherical and spherelike objects in representation theory. This builds on work of Keller \& Reiten \cite{Keller-Reiten} on cluster-tilting theory, which in turn was developed to study cluster algebras. More precisely, $d$-cluster-tilting subcategories in stably $d$-Calabi--Yau categories give rise to $(d+1)$-Calabi--Yau objects in a certain functor category. For $d=1$ or $2$ these objects turn out to be spherical, see \hyref{Proposition}{prop:cluster-sphericals}. Examples include the classical cluster categories of Buan, Marsh, Reineke, Reiten \& Todorov \cite{BMRRT} and the constructions of Gei{\ss}, Leclerc \& Schr{\"o}er, see e.g.\ \cite{Geiss-Leclerc-Schroer}.
We work out the relation to cluster-tilting theory in \hyref{Remark}{rem:cluster-tilting-connection} for the specific algebra $\Lambda_1$ mentioned above.

We study two techniques --- \emph{insertion} and \emph{tacking} --- to modify an algebra in such a way that spherical objects lose the Calabi--Yau property, i.e.\ only become properly spherelike. 
In both constructions, we extend the quiver of an algebra $\Lambda$ by a quiver $\Gamma$, yielding an algebra $\Lambda'$ and an embedding $\Db(\Lambda)\into\Db(\Lambda')$. Moreover, the emerging recollements are similar to blowing-up of varieties in the geometric situation \cite{HKP}; see \hyref{Remark}{rem:algebraic-blowup}.

If $F\in\Db(\Lambda)$ is a spherical object, under explicitly given conditions, $F$ will become properly spherelike in $\Db(\Lambda')$ and then we can compute $\Db(\Lambda')_F$, the spherical subcategory of $F$ in $\Db(\Lambda')$:

\begin{theorem*}[Theorems \ref{thm:insertion-sphericals} \& \ref{thm:tacking-sphericals}]
Let $\Lambda'$ be an algebra which is obtained from $\Lambda$ by tacking on or inserting a quiver $\Gamma$. If $F\in\Db(\Lambda)$ is a spherical object such that $F\in\Db(\Lambda')$ is properly spherelike, then there exists a subquiver $\Gamma'$ of $\Gamma$ such that
 $ \Db(\Lambda')_F = \Db(\Lambda) \oplus \Db(\kk\Gamma') $.
\end{theorem*}

Note that the resulting spherical subcategories are always derived categories of finite-dimensional algebras. While the spherical subcategories of the theorem are simpler than in the geometric case (being direct sums), this result helps computing the posets. As one instance of this, we prove that spherelike posets can get arbitrarily complicated: in \hyref{Lemma}{lem:arbitrary_posets}, we show that any finite poset occurs as a subposet of the spherelike poset of some hereditary algebra.

Moreover, we show that for Vossieck's derived-discrete algebras \cite{Vossieck}, there are spherelike objects such that the spherical subcategory is not the derived category of a finite-dimensional algebra or a projective variety; see \hyref{Proposition}{prop:spherelikes_DDA}. In \hyref{Example}{ex:non-commutative_curve}, we present a spherelike object with the same properties for the algebra $\Lambda_1$ that can be seen as coming from a non-commutative nodal curve.

We present some examples of posets that can occur in spherelike posets, drawing their Hasse diagrams. The left-hand example is the spherelike poset of both a tame hereditary algebra and a derived-discrete algebra; the middle is an instance of \hyref{Lemma}{lem:arbitrary_posets}; the right-hand example is part of a family of hereditary algebras whose spherelike posets contain the given chains.

\noindent
\tikzexample

\begin{question*} 
Is the height of the spherical poset $\poset(\DD)$ bounded?
\end{question*}

\noindent
We note that the rank of the Grothendieck group $K(\DD)$ is a bound for the height of $\poset(\DD)$ in all examples where we can compute this poset completely. One may wonder whether this gives an upper bound in general.
We note that even if true, this bound can be arbitrarily bad, as e.g.\ derived categories of Dynkin quivers always have empty spherelike posets.

\addtocontents{toc}{\protect\setcounter{tocdepth}{1}}   

\subsection*{Conventions on categories}
Throughout, we fix an algebraically closed field $\kk$ and all algebras and varieties are defined over $\kk$. Additive categories are assumed to be $\kk$-linear, and subcategories to be full.

The shift (or translation, or suspension) functor of triangulated categories is denoted by $[1]$. All triangles in triangulated categories are meant to be distinguished. Also, we will generally denote triangles abusively by $A\to B\to C$, hiding the degree increasing morphism $C\to A[1]$.
We write $\Hom^\bullet(A,B) \coloneqq \bigoplus_{i\in\IZ} \Hom(A,B[i])[-i]$ for the homomorphism complexes in triangulated categories; these have zero differential. 

All functors between triangulated categories are meant to be exact. We denote derived functors with the same symbol as the functors between abelian categories. In particular, we write $\otimes_\Lambda$ and $\Hom_\Lambda$ instead of $\smash{\otimes_\Lambda^L}$ and $\RHom_\Lambda$, respectively. Often, we suppress the index $\Lambda$ from $\Hom_\Lambda$.

\subsection*{Conventions on algebras}
Let $\Lambda$ be a finite-dimensional basic algebra. We will denote by $\Db(\Lambda)$ the bounded derived category of finitely generated left $\Lambda$-modules. We denote by $Q(\Lambda)$ the quiver of $\Lambda$, whose vertex set is
  $Q_0(\Lambda) \coloneqq \{ S \in \modules{\Lambda} \text{ simple}\}/\!\cong $
and whose arrows are given by extensions between simples. In other words, $\Lambda$ is given as the path algebra of $Q(\Lambda)$ bound by a (non-unique) ideal. For any $x\in Q_0(\Lambda)$, the associated simple, projective, injective modules are denoted by $S(x), P(x), I(x)$, respectively.
For an idempotent $e\in\Lambda$, we tersely write $\Lambda/e$ for the quotient of $\Lambda$ by the two-sided ideal $\Lambda e\Lambda$ generated by $e$.

The support of a $\Lambda$-module $M$ is
 $\{x\in Q_0(\Lambda) \mid \Hom(P(x),M)\neq0\}$.
Interpreting $M$ as a representation of $Q(\Lambda)$ bound by some ideal, $\supp(M)$ is the set of vertices $x\in Q_0(\Lambda)$ with $\dim(M_x)>0$. The support of an object in $\Db(\Lambda)$ is defined as the union of the supports of all cohomology modules.

The $\kk$-dual of a homomorphism space will be denoted $\Hom(\blank,\blank)^*$. 
If an abstract triangulated category has a Serre functor, we denote it by $\SSS$.
For a finite-dimensional algebra $\Lambda$ of finite global dimension, the Serre functor of $\Db(\Lambda)$ is given by the Nakayama functor $\Naka=\Lambda^* \otimes_\Lambda \blank$, and from \hyref{Section}{sec:quiver-constructions} onwards, we write $\Naka$ instead of $\SSS$. We also use the Auslander--Reiten (AR) translation $\tau=\Naka[-1]$ of $\Db(\Lambda)$. 

$A_n$ denotes the linearly oriented quiver of Dynkin type $A$ with $n$ vertices.

\subsection*{Conventions on recollements}
Repeatedly in this article, we will discuss decompositions of triangulated categories. For this, we use the language of recollements, see e.g.\ \cite{BBD} and \cite{Cline-Parshall-Scott} as general references. Here we follow roughly the survey \cite[\S 2]{Kalck}. Given a triangulated category $\DD$ and a full triangulated sub\-category $\CC$ of $\DD$ such that the inclusion of
 $\CC\orth \coloneqq \{ D\in\DD \mid \Hom^\bullet(\CC,D)=0 \}$ 
has both adjoints, then there is a canonical equivalence
 $\CC \isom \DD/\CC\orth$,
which gives rise to a recollement
\[ \recoll{\CC\orth}{\DD}{\CC} .  \]

We also need a slightly more general notion: if $\CC\into\DD$ has only a right adjoint then $\DD\onto\DD/\CC$ also has only a right adjoint. This right adjoint is explicitly given by the inverse of the canonical equivalence $\CC\orth\isom\DD/\CC$ and the inclusion $\CC\orth\into\DD$. We call this a \emph{weak recollement} and we will write $\wrecoll{\CC\orth}{\DD}{\CC}$.

For readers familiar with the language of semi-orthogonal decompositions, we point out that a weak recollement $\wrecoll{\CC\orth}{\DD}{\CC}$ is equivalent to a weak semi-orthogonal decomposition $\DD = \sod{\CC\orth,\CC}$. If $\CC, \CC\orth$ and $\DD$ have Serre functors, these are recollements and semi-orthogonal decompositions, respectively. See \cite[\S1.1]{Orlov-LG} for details.

\section{Spherelike objects and their spherical subcategories} \label{sec:prelim}

\noindent
We recall some notions and results from our previous article \cite{HKP}.
Let $\DD$ be a Hom-finite triangulated category, and $F$ an object of $\DD$.
\begin{itemize}[leftmargin=1.5em]
\item $F\in\DD$ is said to \emph{have a Serre dual} $\SSS F$, if the cohomological functor
 $\Hom(F,\blank)^* \colon \DD\op \to \modules{\kk}$
is represented by $\SSS F$, where $(\blank)^*$ is the duality of $\kk$-vector spaces. More precisely, $\Hom(F,\blank)\cong\Hom(\blank,\SSS F)^*$. \newline
(If $\DD$ has a Serre functor, then the two meanings of $\SSS F$ coincide. But it can happen that an object has a Serre dual object, even though there is no global Serre functor.)
\item $F$ is called \emph{$d$-Calabi--Yau} if $F[d]$ is a Serre dual for $F$.
\item $F$ is called \emph{$d$-spherelike} if $\Hom^\bullet(F,F) \cong \kk\oplus\kk[-d]$.
\item $F$ is called \emph{$d$-spherical} if it is $d$-spherelike and $d$-Calabi--Yau.
\item $F$ is called \emph{properly $d$-spherelike} if it is $d$-spherelike but not $d$-Calabi--Yau.
\end{itemize}
If the number $d$ is clear from the context, or not relevant, it will be dropped from the notation.

\begin{remark}
There is a dichotomy regarding the algebra structure of $\Hom^\bullet(F,F)$, since a two-dimensional $\kk$-algebra over an algebraically closed field $\kk$ is either isomorphic to $\kk[x]/x^2$ or to $\kk\times\kk$.

The case $\Hom^\bullet(F,F) \cong \kk\times\kk$ can only occur for $d=0$. If $E_1,E_2\in\DD$ are exceptional objects, i.e.\ $\Hom^\bullet(E_i,E_i)=\kk$, such that $\Hom^\bullet(E_1,E_2)=\Hom^\bullet(E_2,E_1)=0$, then $F\coloneqq E_1\oplus E_2$ has endomorphism algebra $\kk\times\kk$. If $\DD$ is idempotent complete, e.g.\ if $\DD$ is the derived category of an abelian category, then all spherelike objects with endomorphism algebra $\kk\times\kk$ are of this kind. We call such spherelike objects \emph{disconnected}; if $\DD$ is idempotent complete, then such an object is decomposable. Correspondingly, spherelike objects $F$ with $\Hom^\bullet(F,F) \cong \kk[x]/x^2$ are called \emph{nilpotent}. We will only be interested in nilpotent spherelike objects.
\end{remark}

\subsection{Spherelike objects as derived invariants}
\label{sec:derived-invariant}

The (non-)existence of a $d$-spherelike object in $\Db(\Lambda)$ is a triangulated invariant of an algebra $\Lambda$. Here, we show how this can be used to distinguish the derived type of algebras. Consider the following three gentle algebras:
\[ \Lambda_1  =
       \kk \Big( \xymatrix@C=1.3em{
          1 \ar@<0.5ex>[r]^a \ar@<-0.5ex>[r]_b &
          2 \ar@<0.5ex>[r]^a \ar@<-0.5ex>[r]_b & 3 }\Big) / (a^2,b^2)
   , \qquad
   \Lambda'_2 =
       \kk \bigg( \raisebox{3ex}{\xymatrix@C=1.6em@R=2ex{
        & \Top \ar@/_0.5pc/[dl]_{\Ardown} \\
 \Left \ar@<-0.5ex>[rr]_\beta \ar@<0.5ex>[rr]^\gamma & & \Right \ar@/_0.5pc/[ul]_{\Arup} }} \bigg) / (\Arup \beta, \gamma \Ardown)
\]
and $\Lambda_2 = \Lambda'_2 / (\Ardown \Arup)$.

\begin{proposition}
\label{prop:non-equivalence}
The triangulated categories $\Db(\Lambda_1)$, $\Db(\Lambda_2)$ and $\Db(\Lambda'_2)$ are pairwise non-equivalent. 
\end{proposition}

\begin{remark}
\label{rem:distinguish}
This has been shown by Bobi\'nski \cite{Bobinski} in greater generality, following work of \cite{Amiot}. That proof gives no direct anwer as to how the categories differ. In \cite{Kalck}, the second author gives an independent proof of \hyref{Proposition}{prop:non-equivalence} which illuminates the difference between the categories somewhat. This inspired the use of spherelike objects shown here.

The algebras $\Lambda_1$ and $\Lambda_2$ have equivalent Euler pairings, and the same Avella-Alaminos--Gei\ss\ invariant; see \cite[\S4]{Kalck} for details. Therefore, no previously known derived invariant distinguishes $\Db(\Lambda_1)$ and $\Db(\Lambda_2)$. The existence of $d$-spherelike objects does: using spherelike posets of \hyref{Section}{sec:poset}, we can concisely state $\posett_{-1}(\Lambda_1) = \varnothing \neq \posett_{-1}(\Lambda_2)$ and $\posett_2(\Lambda_1)  = \varnothing \neq \posett_2(\Lambda'_2)$.
\end{remark}

\begin{proof}
We start by showing that $\Db(\Lambda_2)$ contains $(-1)$-spherelike objects.
Consider the idempotent $e = e_\Top + e_\Left$ and the corresponding corner algebra
\[ e\Lambda_2 e = \kk\Big( \xymatrix@C=2.5em{ \Left \ar@<-0.5ex>[r]_{[\Arup\gamma]} & \Top \ar@<-0.5ex>[l]_{\Ardown} } \Big) / (\Ardown [\Arup\gamma], [\Arup\gamma] \Ardown) . \]
By a direct calculation, the complexes $P(\Top)\to P(\Left)$ and $P(\Left)\to P(\Top)$ are shown to be $(-1)$-spherelike objects of $\Db(e\Lambda_2 e)$. 
There is a fully faithful functor
  $\Lambda_2 e \otimes_{e\Lambda_2 e} \blank \colon \Kb(\projectives{e\Lambda_2 e}) \into \Db(\Lambda_2)$,
and hence the two complexes also give $(-1)$-spherelike objects in $\Db(\Lambda_2)$.
The same reasoning can be applied to the idempotent $e'=e_\Right+e_\Top$ and the complexes $P(\Right) \to P(\Top)$ and $P(\Top) \to P(\Right)$.

Now we claim that $\Db(\Lambda_1)$ does not contain $(-1)$-spherelike objects.
With $\Lambda_1$ of finite global dimension, we can use Happel's equivalence between the derived category and the stable module category over the repetitive algebra: $H\colon \Db(\Lambda_1) \isom \stmodules{\hat\Lambda_1}$. These modules, in turn, can be understood using the combinatorics of strings and bands. Assume that $F\in\Db(\Lambda_1)$ is a $(-1)$-spherelike object. Since band modules have self-extensions, we know that $H(F)$ is a string module. However, the $(-1)$-extension of $H(F)$ forces a 1-extension; this can be shown along the lines of the proof of \cite[Prop.~1.4]{Kalck}. 

Regarding $\Db(\Lambda'_2)$, one can observe that the Euler form of $\Lambda'_2$ is not equivalent to that of $\Lambda_1$ and $\Lambda_2$; for the latter, see \cite[Rem.~1.5]{Kalck}. In particular, the Euler form for $\Lambda_1$ cannot represent 2 (because $(x_1-x_2+x_3)^2\neq2$) whereas the one for $\Lambda'_2$ can, for example by the simple modules $S(\Right)$ and $S(\Top)$. We remark that these are even 2-spherelike and that $P(\Left)$ is 0-spherelike.
\end{proof}

\begin{remark} 
It follows from \cite[Prop.~4.7]{BPP} that every component in the Avella-Alaminos--Gei\ss\ invariant gives rise to a spherelike object. For $\Lambda_1$, $\Lambda_2$, $\Lambda'_2$, these are 3-spherelike objects. In particular, knowing all spherelike objects and the action of the Serre functor on them allows to recover the Avella-Alaminos--Gei\ss\ invariant. There is one exception, namly gentle algebras derived equivalent to path algebras of Dynkin type $A$, but these derived categories are completely characterized by the rank of the Grothendieck group and the absence of indecomposable spherelike objects.
\end{remark}

So an understanding of spherelike objects gives a way to distinguish derived categories. Unfortunately, this does not seem to be as combinatorial as the Avella-Alaminos--Gei\ss{} invariant.

\begin{conjecture}
Let $A$ and $B$ be two connected finite-dimensional gentle algebras. Then the following two statements are equivalent:
\begin{enumerate}
\item $A$ and $B$ are derived equivalent;
\item 
\begin{enumerate}
\item The Euler-forms of $A$ and $B$ are equivalent (in particular, the Grothendieck groups of $A$ and $B$ have the same rank);
\item For every $d \in \IZ$, there are bijections of the sets
\[
\{X \in \Db(A) \mid \text{$X$ is $d$-spherelike} \} 
\longleftrightarrow
\{Y \in \Db(B) \mid \text{$Y$ is $d$-spherelike} \} ,
\]
which commute with the actions of the Serre functors. 
\end{enumerate}
\end{enumerate}
\end{conjecture}

\subsection{Spherical subcategories}

The central idea of \cite{HKP} was to associate a triangulated subcategory to any $d$-spherelike object $F\in\DD$ with a Serre dual in the following fashion: in case $d\neq0$, there is a unique morphism (up to scalars) $F \to \SSS F[-d]$ and we denote its cone by $Q_F$; this object is called the \emph{asphericality} of $F$. (See the Appendix of \cite{HKP} for the case $d=0$.) 
Evidently, $Q_F=0$ if and only if $F$ is $d$-spherical. Next, the \emph{spherical subcategory} $\DD_F$ of $F$ is defined to be the left orthogonal complement of the asphericality:
\begin{equation*}
 \DD_F \coloneqq {}\orth Q_F = \{A\in\DD \mid \Hom^\bullet(A,Q_F)=0\} 
   \quad\text{with}\quad  F \to \SSS F[-d] \to Q_F .
\end{equation*}
The terminology is justified by the results of \cite[Thms.~4.4~\&~4.6]{HKP}:
\begin{itemize}
\item $F\in\DD_F$, and $F$ is a $d$-spherical object of $\DD_F$.
\item If $\UU\subseteq\DD$ is a full triangulated subcategory such that $F\in\UU$, and $F$ is a $d$-spherical object of $\UU$, then $\UU\subseteq\DD_F$.
\end{itemize}
In other words, there is a \emph{unique maximal} triangulated subcategory in which $F$ becomes spherical, and this subcategory is $\DD_F$.
By contrast, Keller, Yang \& Zhou in \cite{Keller-Yang-Zhou} study the minimal subcategory of $\DD$ in which $F$ becomes spherical --- this is the triangulated category $\sod{F}$ generated by $F$.

\begin{remark}
We mention that the functor $\SSS[-d]$ used above to define the asphericality also plays an important role in representation theory \cite{Iyama}: it is the AR-translation for $d=1$ and Iyama's $d$-AR-translation for $d>1$.
\end{remark}

We introduce two conditions, \SOT\ and \SOD, which ensure either better tractability of spherical subcategories or better behaviour of twist functors associated to spherelike objects:

\assumption{\SOT}{
$F$ is $d$-spherical in a subcategory $ \CC \subseteq \DD$ such that the inclusion $\CC \into \DD$ has a right adjoint and $\DD$ a Serre functor.
}

By \cite[Thm.~4.7]{HKP}, condition \SOT\ allows to compute the spherical subcategory without recourse to the asphericality:

\begin{theorem} \label{thm:projection_functor}
If \SOT\ holds then $F$ is $d$-spherelike as an object of $\DD$, and there is a weak recollement
$\wrecoll{\CC\orth\cap {}\orth F}{\DD_F}{\CC}$.
\end{theorem}

In the algebraic examples computed via the theorem in this article, the decomposition happens to be of the simplest type: $\DD_F = (\CC\orth\cap {}\orth F) \oplus \CC $.

\assumption{\SOD}{
$\DD$ has a Serre functor and $\DD_F\into\DD$ has a right adjoint.
}

To any object $F\in\DD$, one wants to associate a twist functor $\TTT_F\colon\DD\to\DD$, as the following cone:
 $\Hom^\bullet(F,\blank)\otimes F \to \id \to \TTT_F$.
Under mild assumptions on $\DD$, this is possible: $\DD$ has to be enhanced (equivalently, algebraic), idempotent complete such that $\dim\Hom^\bullet(A,B)<\infty$ for all objects $A,B\in\DD$. The existence of twist functors in this setting is explained in \cite[\S3.1]{HKP}.

This functor is an autoequivalence if and only if $F$ is spherical. If the condition \SOD\ is satisfied, $\TTT_F$ is still conservative, by \cite[Prop.~4.9]{HKP}, i.e.\ if the twist of a map is an isomorphism, then the map is an isomorphism. Even though we make no use of the twists in this article whatsoever, we check that condition \SOD\ is met in all our examples. 

\medskip

We state a simple observation which explains why there are no negatively-spherical objects in well-behaved algebraic or geometric situations. Recall that a finite-dimensional algebra $\Lambda$ is called \emph{Iwanaga--Gorenstein} if $\Lambda$ has finite injective dimension as a left and a right $\Lambda$-module.

\begin{lemma} \label{lem:negative-spheres}
Let either $\DD = \Db(\Lambda)$ for a finite-dimensional algebra $\Lambda$ of finite global dimension, or $\DD=\Db(X)$ for a smooth, projective variety $X$. If $F\in\DD$ is a $d$-Calabi--Yau object, then $d\geq0$.

Moreover, the same statements hold for $\DD=\Kb(\projectives{\Lambda})$ if $\Lambda$ is an Iwanaga--Gorenstein algebra, and for $\DD=\Perf(X)$ if $X$ is a Gorenstein variety.
\end{lemma}

\begin{remark} 
More generally, let $L\colon\Db(\AA)\to\Db(\AA)$ be the left-derived functor of a right exact endofunctor of an abelian category $\AA$ with enough projectives. 
The proof shows that $LX \cong X[d]$ implies $d \geq 0$.
\end{remark}

\begin{proof}
Assume $d<0$ and let $F\in\DD = \Db(\Lambda)$ be a non-zero object with $\nu F\cong F[d]$. Set $m\in\IZ$ to be the maximal non-zero cohomology of $F$, i.e.\ $F\in\DD^{\leq m}(\Lambda)$ but $F\notin \DD^{<m}(\Lambda)$.
Computing $\nu F$ using a projective resolution, we see that again $\nu F \in \DD^{\leq m}(\Lambda) \subset \DD^{<m-d}(\Lambda)$. On the other hand, $F[d]\in\DD^{\leq m-d}(\Lambda)$ and $F[d]\notin\DD^{<m-d}(\Lambda)$, a contradiction.

For $\DD=\Db(X)$, the Serre functor is given by $\SSS(F) = F\otimes\omega_X[\dim X]$. Since tensoring with a line bundle is exact, we see that $\SSS(F)\cong F[d]$ if and only if $d=\dim X$ and $F\cong F\otimes\omega_X$.

The generalisation to Gorenstein algebras and varieties follows with the same proof, observing that $\Kb(\projectives{\Lambda})$ and $\Perf(X)$ have Serre functors, again given by the Nakayama functor and $\blank\otimes\omega_X[\dim X]$, respectively.
\end{proof}

These similarities notwithstanding, there is an important difference between the algebraic and the geometric cases: an algebra can have $d$-spherical objects for different $d$. One example is the derived-discrete algebra $\Lambda(1,2,0)$; see \hyref{Section}{sec:DDA}. This cannot happen in $\Db(X)$ if $X$ is a connected and smooth variety of dimension $n$: the Serre functor of $\Db(X)$ is $\omega_X\otimes(\blank)[n]$ and, crucially, the tensor product by the line bundle $\omega_X$ is an exact functor on $\Coh(X)$; hence any Calabi--Yau object in $\Db(X)$ is necessarily $n$-CY.

\begin{lemma} \label{lem:smooth_spherical_subcategories}
Let either $\DD = \Db(\Lambda)$ for a finite-dimensional algebra $\Lambda$ of finite global dimension, or $\DD=\Db(X)$ for a smooth, projective variety $X$. If $F\in\DD$ is a spherelike object such that $\DD_F$ is equivalent to $\Db(\Lambda')$ for a finite-dimensional algebra $\Lambda'$ or to $\Db(X')$ for a projective variety $X'$, then $\Lambda'$ has finite global dimension and $X'$ is smooth, respectively.
\end{lemma}

\begin{proof}
Consider the full subcategory of $\DD$ consisting of homologically finite objects
 $ \DD_\hf \coloneqq \{D\in\DD \mid s(\Hom^\bullet(D,D')) \text{ is finite for all } D'\in\DD\} $
where, for a complex $V^\bullet\in\Db(\kk)$ of vector spaces, $s(V^\bullet) \coloneqq \{i\in\IZ \mid H^i(V^\bullet)\neq0\}$.

By \cite[Prop.~1.11]{Orlov}, $\Perf(Y) \cong \Db(Y)_\hf$ for possibly singular varieties. Likewise, if $A$ is a finite-dimensional algebra of arbitrary global dimension, then $\Perf(A) = \Db(A)_\hf$. The inclusion $\Perf(A)\subseteq\Db(A)_\hf$ is easy; and if $M\in\Db(A)$ but $M\notin\Perf(A)$, then we can replace $M$ by a bounded-above, minimal projective resolution $P^\bullet\to M$. Let $S\coloneqq \bigoplus_{i\in Q_0(A)} S(i)$ be the sum of all simple $A$-modules. Then there are sufficiently negative numbers $i\ll0$ with cohomology modules $H^i(P^\bullet)\neq 0$, as $M\cong P^\bullet$ is not perfect. For each $0\neq H^i(P^\bullet)$, there is some non-zero map $H^i(P^\bullet) \to S$, giving rise to a map $P^\bullet \to S[i]$ which is not null-homotopic, hence $M\notin\Db(A)_\hf$.

With $\DD=\DD_\hf$ by our assumption on $X$ and $\Lambda$, the subcategory $\DD_F$ contains only homologically finite objects as well, hence the claim.
Note this proof only uses that $\DD_F$ is a triangulated subcategory of $\DD$.
\end{proof}

\begin{remark}
In \hyref{Propostion}{prop:spherelikes_DDA} we will meet spherical subcategories of $d$-spherelike objects with $d<0$, hence these categories are not of the form $\Db(X)$ or $\Db(\Lambda)$ by the above lemmata.
In \hyref{Section}{sec:noncomm-curve}, another spherical subcategory with this property is presented but for a $3$-spherelike object.
\end{remark}

\section{A new triangulated invariant: the spherelike poset}
\label{sec:poset}

\noindent
Let $\DD$ be a $\kk$-linear Hom-finite triangulated category. We define the sets
\begin{align*}
 \poset(\DD)    &\coloneqq \{ \DD_F \mid F\in\DD \text{ spherelike, nilpotent and has a Serre dual}\}, \\
 \posett_d(\DD) &\coloneqq \{ \DD_F \in \poset(\DD) \mid F\in\DD \text{ $d$-spherelike}\} .
\end{align*}
The set $\poset(\DD)$ is partially ordered by inclusion and called the \emph{spherelike poset} of $\DD$. The subposet $\posett_d(\DD)$ is called the \emph{$d$-spherelike poset}. We also write $\poset(\Lambda)=\poset(\Db(\Lambda))$ and $\poset(X)=\poset(\Db(X))$, if $\DD$ is the bounded derived category of an algebra $\Lambda$ or a variety $X$, respectively.
If there is a nilpotent spherical object $F\in\DD$, then $\DD=\DD_F\in\poset(\DD)$ is the maximal element.

\begin{remark}
A fully faithful functor $\iota\colon \DD\into\DD'$ maps spherelike objects to spherelike objects. However, the assignment $\DD_F \mapsto {\DD'}_{\iota(F)}$ does not induce a well-defined map of sets $\poset(\DD)\to\poset(\DD')$ in general.

An example is given by spherical objects $F_1,F_2\in\DD$ such that $\iota(F_1)\in\DD'$ remains spherical but $\iota(F_2)\in\DD'$ becomes properly spherelike. A concrete instance of this is provided by derived categories of derived-discrete algebras $\DD=\Db(\Lambda(1,2,0))$ and $\DD'=\Db(\Lambda(1,4,0))$; see \hyref{Section}{sec:DDA}.

Note that $\DD_F \mapsto \iota(\DD_F)$ will in general not work either: while any functor preserves inclusions of subcategories, it can happen that $\iota(\DD_F)\notin\poset(\DD')$. Indeed, the very same example
 $\iota\colon \DD=\Db(\Lambda(1,2,0)) \into \DD'=\Db(\Lambda(1,4,0))$
shows this: $\iota(\DD_{F_1}) = \iota(\DD_{F_2}) = \iota(\DD) \subsetneq {\DD'}_{\iota(F_1)}, {\DD'}_{\iota(F_2)}$
by \hyref{Proposition}{prop:spherelikes_DDA}.
\end{remark}

Nonetheless, the next lemma shows that spherical subcategories are well-behaved with respect to autoequivalences.

\begin{lemma} \label{lem:sphericals_under_equivalences}
Let $\varphi\colon \DD\isom\DD'$ be an equivalence of triangulated categories and let $F\in\DD$ be a spherelike object having a Serre dual.
Then $\varphi(F)\in\DD'$ is spherelike with $\DD'_{\varphi(F)} = \varphi(\DD_F)$.

In particular, the spherelike poset is an invariant of $k$-linear Hom-finite triangulated categories.
\end{lemma}

\begin{proof}
We start by showing that the asphericality objects behave well under equivalences: $\varphi(Q_F)\cong Q_{\varphi(F)}$, so $\varphi(\DD_F)$ is indeed in $\poset(\DD')$. To begin with, we have assumed that there exists a Serre dual object $\SSS F$ for $F$. Then $\varphi F$ has a Serre dual object as well, and in fact, $\SSS(\varphi F) \cong \varphi(\SSS F)$, as follows from
  $\Hom(\varphi F,A)^* = \Hom(F,\varphi\inv A)^* = \Hom(\varphi\inv A,\SSS F) = \Hom(A,\varphi F)$.
Now we can apply $\varphi$ to the defining triangle $F\to\SSS F[-d]\to Q_F$, to get:
\[ \xymatrix{
 \varphi(F) \ar[r] \ar@{=}[d] & \varphi(\SSS F)[-d] \ar[r] \ar[d]^\cong & \varphi(Q_F) \ar[r] \ar@{-->}[d] & \varphi(F)[1] \ar@{=}[d] \\
 \varphi(F) \ar[r]            & \SSS(\varphi F)[-d] \ar[r]             & Q_{\varphi(F)} \ar[r]              & \varphi(F)[1]            
} \]
where the vertical isomorphism and commutativity of the left-hand square follow from the above computation with $A = \varphi(\SSS F)[-d]$, and all those Hom spaces having dimension 1.
The dashed arrow exists by the axioms of triangulated categories, and is an isomorphism by the five lemma.
Hence $\DD'_{\varphi(F)} = {}\orth Q_{\varphi(F)} = {}\orth (\varphi(Q_F)) = \varphi({}\orth Q_F) = \varphi(\DD_F)$.
\end{proof}

By the lemma, an equivalence $\varphi\colon \DD \isom \DD'$ induces a well-defined map of posets
 $\poset(\DD)\to\poset(\DD'), \DD_F\mapsto {\DD'}_{\varphi(F)} = \varphi(\DD_F)$.
Hence it makes sense to look at the spherelike poset up to auto\-equivalences; we define the \emph{stable spherelike poset} of $\DD$ to be
\[ \posetaut(\DD) \coloneqq \poset(\DD) / \! \Aut(\DD) . \]
Analogously, the \emph{stable $d$-spherelike poset} is $\posetautt_d(\DD) \coloneqq \posett_d(\DD) / \! \Aut(\DD)$.

Having introduced the poset $\poset(\DD)$ as a triangulated invariant of $\DD$, we obtain further numerical invariants: the \emph{cardinality}, the \emph{height} and the \emph{width} of $\poset(\DD)$ and its variants $\posetautt_d(\DD)$ etc. All of these take values in $\IN\cup\{\infty\}$. We recall that the height of a poset is the maximum among lengths of chains, i.e.\ subsets consisting of pairwise comparable elements. Dually, the width of a poset is the maximal number of elements of antichains, i.e.\ subsets of pairwise incomparable elements.

\hyref{Lemma}{lem:arbitrary_posets} shows that spherelike posets can become very complicated, already for hereditary algebras.
\hyref{Corollary}{cor:increasing_heights} exemplifies this with a series of algebras of increasing heights, and \hyref{Example}{ex:poset_cycle} uses the method of the lemma to create a spherelike poset containing a cycle. Moreover, \hyref{Example}{ex:infinite_width} has an algebra of infinite width, even up to autoequivalences.

\begin{example}[Spherelike posets are non-additive invariants]
Let $\DD$ and $\DD'$ be two triangulated categories. For the spherelike poset of the direct sum $\DD\oplus\DD'$, two different cases occur: if neither $\DD$ nor $\DD'$ contain spherical objects, then
 $\poset(\DD\oplus\DD') \cong \poset(\DD) \amalg \poset(\DD')$.
On the other hand, if $\DD$ or $\DD'$ contain spherical objects, then the Hasse diagram of $\poset(\DD\oplus\DD')$ looks like
\[ \xymatrix@C=-0.5em@R=2ex{
      & \DD\oplus\DD' \ar@{-}[dl] \ar@{-}[dr] \\
 \poset(\DD)\setminus\{\DD\} & & \poset(\DD')\setminus\{\DD'\}
} \]
where (at most) one of the two set differences may be trivial. This kind of non-additivity is atypical for triangulated invariants such as K-theory.
\end{example}

As a first concrete example, we mention that if $\DD$ is a Calabi--Yau category, i.e.\ the Serre functor is isomorphic to a shift, then either $\poset(\DD)=\{\DD\}$ (if $\DD$ contains some spherical object) or else $\poset(\DD)=\varnothing$. This applies to cluster categories, or $\Perf(\Lambda)$ for a symmetric algebra $\Lambda$, or to $\Db(X)$ for smooth, projective varieties $X$ with trivial canonical bundle.

If $C$ is a smooth projective curve, then $\poset(C)=\posett_1(C)=\{\Db(C)\}$. The spherical poset becomes more interesting for the generalisation to weighted projective lines, see \hyref{Section}{sec:canonical_algebras}, and for surfaces, see \hyref{Example}{ex:infinite_width}.

\begin{example}[Hereditary algebras] \label{ex:hereditary}
Let $\Lambda \coloneqq \kk Q$ be the path algebra of an acyclic quiver. Then $\Lambda$ is hereditary, i.e.\ $\gldim(\Lambda)=1$ and therefore every object in $\Db(\Lambda)$ is isomorphic to its (graded) cohomology. If $Q$ is Dynkin, then every indecomposable object in $\modules{\Lambda}$ is exceptional and therefore there are no nilpotent spherelike objects in $\Db(\Lambda)$; in particular, $\poset(\Lambda)=\varnothing$.

If $Q$ is Euclidean, i.e.\ of type $\tilde A_n, \tilde D_n, \tilde E_n$, then $\Lambda$ is tame hereditary. For types $\tilde D_n$ and $\tilde E_n$, $\modules{\Lambda}$ has three non-homogeneous tubes of ranks $p,q,r$ and $p+q+r-2=n$ and assuming $p\leq q\leq r$, we have
\[ \poset(\Lambda) = \{\DD > \DD_{X_1},\ldots,\DD_{X_p},\DD_{Y_1},\ldots,\DD_{Y_q},\DD_{Z_1},\ldots,\DD_{Z_r}\}, \]
where $X_i, Y_j, Z_k$ are the indecomposable modules of quasi-lengths $p,q,r$ in the exceptional tubes, respectively. These modules are properly 1-spherelike. See \hyref{Section}{sec:canonical_algebras} for a description of the Auslander--Reiten quiver of tubes, in the setting of torsion sheaves over a weighted projective line.
 In type $\tilde A_{p,q}$, the category $\modules{\Lambda}$ has up to two non-homogeneous tubes of ranks $p$ and $q$ and $\poset(\kk\tilde A_{p,q})$ looks similar, with the spherical subcategories $\DD_{Z_k}$ omitted and special cases for $p\leq1$ or $p=q=1$.

The modules $X_i$ are in the same $\tau$-orbit, and hence identified in the poset $\posetaut(\Lambda)$. The same is true for the modules $Y_j$ and $Z_k$, respectively.  Moreover, all quasi-simple modules in homogenous tubes are 1-spherical. In particular, $\poset(\Lambda)$ and $\posetaut(\Lambda)$ have height 2, except for $Q=\tilde A_1$. In \hyref{Section}{sec:canonical_algebras}, we treat the more general Geigle--Lenzing weighted projective lines, which include the categories considered above for the weight sequences $(p,q,r)$.

If $Q$ is an $n$-Kronecker quiver with $n \geq 3$, then using the Euler form one can check that $\Db(\Lambda)$ has no spherelike objects at all. We don't have a description of the poset for general wild hereditary algebras. However, if $Q$ contains a full Euclidean subquiver, then $\poset(\Lambda)\neq\varnothing$. Moreover, there are acyclic quivers $Q$ such that the corresponding poset has height $\geq n$ for an arbitrary integer $n$, see \hyref{Corollary}{cor:increasing_heights}.
\end{example}

\section{Two quiver constructions} \label{sec:quiver-constructions}

\noindent
We present two general constructions that can be applied to quiver algebras. In one of them, we insert a linearly oriented quiver of Dynkin type $\A{n}$ at a specified vertex. In the other, we tack an arbitrary quiver without relations to an algebra at a sink. In both cases, we get recollements for the derived categories of the resulting algebras. In particular, a spherical object for the original algebra gives rise to a spherelike object for the amalgamated algebra. These constructions will turn up in our examples.

\subsection*{Idempotent calculus}
It is well-known that idempotents are sources of recollements. For easier reference, we introduce the following notion:

\begin{definition}
  Let $\Lambda$ be a finite-dimensional algebra. An idempotent element $e\in\Lambda$ is called \emph{recollant} if
\begin{itemize}
\item $\Ext^k_\Lambda(\Lambda/e,\Lambda/e) = 0$ for $k>0$;
\item $\projdim {}_\Lambda(\Lambda/e)< \infty$;
\item $\projdim (\Lambda/e)_\Lambda< \infty$ or $\projdim (\Lambda e)_{e\Lambda e}< \infty$.
\end{itemize}
\end{definition}

\begin{proposition}[{\cite[Thm.\ 1, Prop.\ 2]{Miyachi}}]  \label{prop:recollement}
Let $\Lambda$ be a finite-dimensional algebra, and let $e\in\Lambda$ be a recollant idempotent. Then there is a recollement
\[ \xymatrix@C=10ex{
\Db(\Lambda/e) \ar[r]|{\,i_*=i_!\,} &
\Db(\Lambda) \ar@/_3ex/[l]|{\,i^*\,} \ar@/^3ex/[l]|{\,i^!\,} \ar[r]|{\,j^!=j^*\,} &
\Db(e\Lambda e) \ar@/_3ex/[l]|{\,j_!\,} \ar@/^3ex/[l]|{\,j_*\,}
} \]
where the involved functors are

\vspace{0.5ex}
\scalebox{0.80}{$
  \begin{array}{@{} r@{\:}r@{\:}l @{\hspace{3em}} r@{\:}r@{\:}l @{}}
                   & i^* = & \Lambda/e \otimes_\Lambda\blank &     \jincl \coloneqq & j_! = & \Lambda e \otimes_{e\Lambda e}\blank      \\
  \iincl \coloneqq & i_* = & \Hom_{\Lambda/e}(\Lambda/e,\blank) = \Lambda/e \otimes_{\Lambda/e} \blank = i_! &
     \pi \coloneqq & j^! = & \Hom_{\Lambda}(\Lambda e,\blank)   =  e \Lambda \otimes_{\Lambda} \blank = j^*                             \\      
                   & i^! = & \Hom_\Lambda(\Lambda/e,\blank)    &                      & j_* = & \Hom_{e\Lambda e}(e \Lambda,\blank)     \\

\end{array}$}
\end{proposition}

We use the following notation throughout the article:
\begin{align*}
\jincl \coloneqq
\Lambda e\otimes_{e\Lambda e}(\blank)     & \colon \Db(e\Lambda e) \into \Db(\Lambda) , \\
\iincl \coloneqq
\Hom_{\Lambda/e}(\Lambda/e,\blank)        & \colon \Db(\Lambda/e) \into \Db(\Lambda) , \\
\pi \coloneqq
\Hom_{\Lambda}(\Lambda e,\blank)
 =  e \Lambda \otimes_{\Lambda} (\blank)  & \colon \Db(\Lambda) \to \Db(e\Lambda e) .
\end{align*}

\begin{remark}
The recollement exists in greater generality, if we replace the left-hand category $\Db(\Lambda/e)$ by $\smash{\Db_{\Lambda/e}(\Lambda)}$, i.e.\ the full subcategory of $\Db(\Lambda)$ whose objects have cohomology in $\modules{\Lambda/e}$; see \cite[Sec.~2]{Cline-Parshall-Scott}.
More precisely, the objects in $\Db_{\Lambda/e}(\Lambda)$ are supported off $e$ which means that
\[ M\in\Db_{\Lambda/e}(\Lambda) \iff \supp(e)\cap\supp(M)=\varnothing , \]
where $\supp(e) \coloneqq \{x \in Q_0(\Lambda) \mid e \cdot S(x) = S(x) \}$.

Finally, we want to note that $\smash{\Db_{\Lambda/e}(\Lambda) \cong \thick_\Lambda(\modules {\Lambda/e})}$, so in the above situation the image of $\iincl$ is the subcategory $\thick_\Lambda(\modules {\Lambda/e})$. On the other hand, the image of $\jincl$ is $\thick_\Lambda(\Lambda e)$.
\end{remark}

\begin{proposition}[{\cite[Lem.~2.1]{Wiedemann}}]
\label{prop:recoll-finite-gldim}
Let $\Lambda$ be a finite-dimensional algebra, and let $e\in\Lambda$ be a recollant idempotent.
Then $\Lambda$ has finite global dimension if and only if $\Lambda/e$ and $e\Lambda e$ have.
\end{proposition}

Recall some basic homological notions: an object $E$ of a triangulated category $\DD$ is \emph{exceptional} if $\Hom^\bullet(E,E)=\kk$. A sequence of objects $E_1,\ldots,E_r$ is an \emph{exceptional sequence} if each $E_i$ is exceptional and $\Hom^\bullet(E_i,E_j)=0$ if $i>j$. The exceptional sequence is \emph{full} if it generates $\DD$; this is denoted by $\DD=\sod{E_1,\ldots,E_r}$. If the sequence is not full, then $\sod{E_1,\ldots,E_r}$ denotes the thick subcategory of $\DD$ generated by it.
The next fact is well-known.

\begin{lemma} \label{lem:exseq}
Let $\Lambda$ be a finite-dimensional algebra such that the quiver $Q(\Lambda)$ has no oriented cycles.
Then there is an ordering of $Q_0(\Lambda)=\{x_1,\ldots,x_t\}$ such that there is a full exceptional sequence $\Db(\Lambda) = \sod{ S(x_1),\ldots,S(x_t) }$.
\end{lemma}

\begin{proof}
It is well-known that $\Hom(S(x),S(y)) \cong \delta_{xy} \cdot \kk$.
Moreover for $k>0$, $\Ext^k(S(x),S(y))$ vanishes if there is no path $x \to \cdots \to y$. Since there are no oriented cycles in $Q(\Lambda)$, we can order the simples to form an exceptional sequence. It is clear that they generate the whole category, so it is a full exceptional sequence.
\end{proof}

\begin{proposition}
\label{prop:idempotent-sod}
Let $\Lambda$ be a finite-dimensional algebra having finite global dimension with a recollant idempotent $e\in\Lambda$, such that the quiver $Q(\Lambda/e)$ has no oriented cycles.
Let $F \in \Db(e\Lambda e)$ be a $d$-spherical object.
If the simple modules of the full exceptional sequence
  $\Db(\Lambda/e) = \sod{ S(x_1),\ldots,S(x_t) }$
satisfy
\[ 
\Hom^\bullet_{e\Lambda e}(F,\pi \Naka S(x_i)) = 0 \text{ for } i=1,\ldots,t-1, 
\]
then the following dichotomy holds:
\begin{enumerate}
\item If $\Hom^\bullet_{e\Lambda e}(F,\pi \Naka S(x_t)) = 0$, then
      $\jincl F \in \Db(\Lambda)$ is $d$-spherical.
\item If  $\Hom^\bullet_{e\Lambda e}(F,\pi \Naka S(x_t)) \neq 0$, 
      then $\jincl F \in \Db(\Lambda)$ is properly $d$-spherelike, 
      and its spherical subcategory $\Db(\Lambda)_{\jincl F}$ has the
      weak recollement
\[  
\wrecoll{ \sod{ S(x_1),\ldots,S(x_{t-1}) } }
        { \Db(\Lambda)_{\jincl F} }
        { \Db(e \Lambda e) } .
\]
\end{enumerate}
\end{proposition}

\begin{proof}
By our assumption on $\Lambda$, the conditions of \hyref{Proposition}{prop:recollement} are met, yielding a recollement
$\recoll{ \Db(\Lambda/e) }{ \Db(\Lambda) }{ \Db(e\Lambda e) }$.
And since $Q(\Lambda/e)$ has no oriented cycles, \hyref{Lemma}{lem:exseq} provides a full exceptional sequence $\Db(\Lambda/e) = \sod{ S(x_1),\ldots,S(x_t) }$.
By abuse of notation, we will denote the simple $\Lambda$-modules $\iota S(x_i)$ again by $S(x_i)$. 

By \hyref{Theorem}{thm:projection_functor}, the recollement for $\Db(\Lambda)$ restricts to
\begin{equation} 
\label{prop:eq:recollement-adjunction-duality}
    \wrecoll{ \sod{ S(x_1),\ldots,S(x_t) } \cap \lorth\jincl F }
            { \Db(\Lambda)_{\jincl F} }
            { \Db(e \Lambda e) }
. 
\end{equation}
Now the functor $\pi = \Hom_\Lambda(\Lambda e, \blank)$ is right adjoint to $\jincl = \Lambda e \otimes_{e \Lambda e} \blank$, so
\[  \Hom^\bullet_\Lambda(S(x_i),\jincl F) = \Hom^\bullet_\Lambda(\jincl F,\Naka S(x_i))^* = \Hom^\bullet_{e\Lambda e}(F,\pi \Naka S(x_i))^* .  \]
Hence the final hypothesis translates to
 $\Hom^\bullet_\Lambda(S(x_i),\jincl F) = 0$ for $i<t$.

If also $\Hom^\bullet_\Lambda(S(x_t),\jincl F)$ vanishes, 
then the weak recollement of $\Db(\Lambda)_{\jincl F}$ in \eqref{prop:eq:recollement-adjunction-duality} is just the restriction of the recollement of $\Db(\Lambda)$. 
In particular, $\Db(\Lambda)_{\jincl F} = \Db(\Lambda)$, so $\jincl F$ is $d$-spherical in $\Db(\Lambda)$.

Assume now that $\Hom^\bullet_\Lambda(S(x_t),\jincl F) \neq 0$.
We will show that
\[
\sod{S(x_1),\ldots,S(x_{t-1})} \subseteq \sod{ S(x_1),\ldots,S(x_t)} \cap \lorth\jincl F
\]
is an equality which implies, in particular, that $\jincl F$ is properly $d$-spherelike.
For the reverse inclusion, consider $M \in \sod{S(x_1),\ldots,S(x_t)} \cap \lorth\jincl F$.
Its left mutation at $S(x_t)$ is $M'\in \sod{S(x_1),\ldots,S(x_{t-1})}$ defined by the triangle $M' \to \Hom^\bullet(S(x_t),M)\otimes S(x_t) \to M$. Applying $\Hom^\bullet(\blank,\jincl F)$ to this triangle yields
$0 = \Hom^\bullet(S(x_t),M) \otimes \Hom^\bullet(S(x_t),\jincl F),$
because of $M\in \lorth\jincl F$ and $M'\in \sod{S(x_1),\ldots,S(x_{t-1})} \subseteq \lorth\jincl F$.
The assumption $\Hom^\bullet(S(x_t),\jincl F) \neq 0$ finally forces
$\Hom^\bullet(S(x_t),M) = 0$, so that $M\in\sod{S(x_1),\ldots,S(x_{t-1})}$.

Plugging the equality
  $\sod{S(x_1),\ldots,S(x_{t-1})} = \sod{ S(x_1),\ldots,S(x_t)} \cap \lorth\jincl F$
into the recollement \eqref{prop:eq:recollement-adjunction-duality} proves the second part of the proposition.
\end{proof}

Many of our examples use this duality-adjunction argument, and therefore deal with computing $\pi\Naka$ on the simples killed by $e$.

\addtocontents{toc}{\protect\setcounter{tocdepth}{2}}  
\subsection{$\A{n}$-insertion}

Let $\Lambda=\kk Q/I$ be a finite-dimensional algebra, given by a quiver $Q$ bound by an admissible ideal $I$. We fix a vertex $x\in Q_0$ and a number $n\in\IN$. The \emph{$\A{n}$-insertion} in $\Lambda=\kk Q/I$ at the vertex $x$ is the path algebra $\lineinsertion{\Lambda}{n}{x} \coloneqq \kk Q(nx)/I(nx)$ where
\begin{itemize}
\item $Q(nx)_0 \coloneqq \bigl(Q_0 \setminus \{x\}\bigr) \cup \{x_0,\ldots,x_n\}$ on vertex sets;
\item the vertex $x$ is replaced by the quiver $x_0 \to \cdots \to x_n$ of type $\A{n+1}$;
\item arrows going into $x$ become arrows going into $x_0$, and arrows going out of $x$ become arrows going out of $x_n$. 
\end{itemize}
The relations $\rho(nx) \in I(nx)$ are obtained from the relations $\rho \in I$ by the following modification procedure
\begin{itemize}
\item if a relation is not incident with the vertex $x$, it remains unchanged;
\item if a relation passes $x$, extend it by $\xi \coloneqq x_0 \to \cdots \to x_n$;
\item if a relation starts in $x$, let it start in $x_n$;
\item if a relation ends in $x$, let it end in $x_0$.
\end{itemize}
More precisely, an element $\rho \in I$ can be written as $\rho= \sum_i \lambda_i p_{1,i}\cdots p_{m_{i},i}$ where the $p_{j,i}$ are \emph{non-trivial} paths in $Q$ which do not contain subpaths of the form $ \to x \to$ and all of which but $p_{1,i}$ (the last one of each summand) end in $x$. Then
\[ \rho(nx) \coloneqq \sum_i \lambda_i p_{1,i} \xi p_{2 i} \xi \cdots \xi p_{m_{i}-1,i}  \xi p_{m_{i},i}  . \]
It suffices to apply this procedure to a set of generators of $I$.

\begin{remark}
This construction only depends on the $\kk$-algebra $\Lambda$, i.e.\ 
an isomorphism of bound path algebras $\varphi\colon\Lambda = \kk Q/I \isom \kk Q'/I' = \Lambda'$ gives rise to an isomorphism $\tilde\varphi\colon \kk[Q(nx)]/I(nx) \isom \kk[Q'(nx')]/I'(nx')$, where $x' \coloneqq \varphi(x)$.

To see this, by \cite[\S II.3]{ASS} the quivers $Q$ and $Q'$ can be identified and $\varphi(I)=I'$. So we can define $\tilde\varphi$ by extending $\varphi$, sending $x_i \to x_{i+1}$ to $x'_i \to x'_{i+1}$. 
As a zero relation $\tilde\rho = \rho(nx) \in I(nx)$ is then sent to $\tilde\varphi(\tilde\rho) = \phi(\rho)(nx') \in I'(nx')$, the morphism $\tilde\varphi$ is well-defined.
Moreover, a $\tilde\rho' \in I'(nx')$ is of the form $\rho'(nx')$ with $\rho' = \varphi(\rho') \in I'$, hence $\varphi(\rho(nx))=\tilde\rho'$. So $\tilde\varphi$ is an algebra homomorphism, which is an injection of finite dimensional $\kk$-vector spaces of the same dimension, hence an isomorphism of algebras.
\end{remark}
 
This remark justifies the notation $\Lambda(nx)$.
Note that $\modules{\Lambda(nx)}$ has $n$ additional isomorphism classes of simple modules.

\begin{example}
We consider the Auslander algebra $\Lambda$ of $\kk[x]/x^2$ given by

 $\xymatrix{0 \ar@<0.6ex>[r]|{\,a\,} & 1 \ar@<0.6ex>[l]|{\,b\,} }$
with the relation $ba=0$. Inserting $\A{2}$ at 0 yields the quiver shown left, whereas inserting $\A{2}$ at 1 yields the quiver on the right:
\[ \begin{array}{c}
\xymatrix@R=2ex{
 \bullet \ar[dr]^a \\
 \bullet \ar[u]  & 1 \ar[ld]^b  \\ 
 \bullet \ar[u]
} \\
\text{with }  ba=0
\end{array} 
\qquad \qquad
\begin{array}{c}
 \xymatrix@R=2ex{
           & \bullet \ar[d]^{\alpha}   \\
 0 \ar[ru]^a & \bullet \ar[d]^{\beta} \\ 
           & \bullet \ar[ul]^b
}\\
\text{with } b\beta\alpha a =0
\end{array} 
\]
\end{example}

\begin{example}
The derived-discrete algebra $\Lambda(r,n,m)$ of \hyref{Section}{sec:DDA} is an $\A{n-r-1}$-insertation of $\Lambda(r,r+1,m)$ at the successor of the trivalent vertex.
\end{example}

We choose the idempotent $e \coloneqq 1 - e_{x_0} - e_{x_1} -\ldots - e_{x_{n-1}} \in \lineinsertion{\Lambda}{n}{x}$,
and we will consider the induced functor
\[  \jincl \coloneqq \lineinsertion{\Lambda}{n}{x} e\otimes_\Lambda\blank \colon 
                     \Db(\Lambda) \to \Db(\lineinsertion{\Lambda}{n}{x}) .  \]
This choice of idempotent is convenient in our computations below. However, the theory also works for any other choice of idempotent
$f_i \coloneqq 1 - e_{x_0} - e_{x_1} -\ldots - e_{x_{i-1}} - e_{x_{i+1}} - \ldots - e_{x_n}$ with $i=0, \ldots, n-1$, which yields a different inclusion $\lineinsertion{\Lambda}{n}{x} f_i\otimes_\Lambda\blank \colon \Db(\Lambda) \to \Db(\lineinsertion{\Lambda}{n}{x})$.

\begin{remark}
\label{rem:ChenKrause}
Our construction is related to the \emph{abelian expansion} of Chen \& Krause \cite{ChenKrause}.
On the level of module categories we have an (underived)
$\jincl_i = \lineinsertion{\Lambda}{n}{x} f_i \otimes_\Lambda\blank \colon \modules{\Lambda} \to \modules{\lineinsertion{\Lambda}{n}{x}}$. For $n=1$, one can show that $\jincl_0 \colon \modules{\Lambda} \to \modules{\lineinsertion{\Lambda}{1}{x}}$ is an abelian expansion in  the sense of Chen \& Krause; see \cite[Ex.~3.6.2]{ChenKrause}.


We note that the convention in our article is to work with the idempotent $e=f_n$. In general, the inclusion $\jincl_1 \colon \modules{\Lambda} \to \modules{\lineinsertion{\Lambda}{1}{x}}$ defined by this idempotent $e$ is not an expansion of abelian categories; see \cite[Prop.~3.2.2(1)]{ChenKrause}. Note that the only difference between $\jincl_0$ and $\jincl_1$ is the way how $\modules{\Lambda}$ is embedded into $\modules{\lineinsertion{\Lambda}{1}{x}}$.

%
\end{remark}

\begin{lemma} \label{lem:insertion-basics}
Let $\lineinsertion{\Lambda}{n}{x}$ be the $\A{n}$-insertion in $\Lambda$ at $x\in Q_0(\Lambda)$.
Then
\begin{enumerate}
\item $\Lambda \cong e\lineinsertion{\Lambda}{n}{x}e$ and $\lineinsertion{\Lambda}{n}{x}/e \cong \kk\A{n}$;
\item if $\projdim (\Lambda(nx)e)_{\Lambda} < \infty$, then $e$ is a recollant idempotent and so there is a recollement
      \[\recoll{\Db(\Lambda(nx)/e)}{\Db(\Lambda(nx))}{\Db(\Lambda)};\]
\item if $\Lambda$ has finite global dimension, then so does $\lineinsertion{\Lambda}{n}{x}$;
\item the functor
      $\iincl = \Hom_{\lineinsertion{\Lambda}{n}{x}/e}(\lineinsertion{\Lambda}{n}{x}/e,\blank) \colon
                \Db(\Lambda(nx)/e) \into \Db(\Lambda(nx))$
      in (2) is fully faithful with image $\sod{ S(x_0),\ldots,S(x_{n-1})}$.
      Each $S(x_i)$ is exceptional as an object of $\Db(\lineinsertion{\Lambda}{n}{x})$.
\end{enumerate}
\end{lemma}

\begin{proof}
(1) 
The isomorphism $\lineinsertion{\Lambda}{n}{x}/e \cong \kk\A{n}$ is clear.
To show $\Lambda \cong e\lineinsertion{\Lambda}{n}{x}e$, we note that 
\[
e(\kk \lineinsertion{Q}{n}{x}/\lineinsertion{I}{n}{x})e \cong (e\kk \lineinsertion{Q}{n}{x}e)/(e\lineinsertion{I}{n}{x}e)
\]
By construction, we get a bijection between the quiver of $e\kk \lineinsertion{Q}{n}{x}e$ and the original quiver $Q$ by sending $x_n$ to $x$, which extends to an isomorphism $\varphi\colon e\kk \lineinsertion{Q}{n}{x}e \cong \kk Q$ of algebras.

To establish the claim, we check that $e\lineinsertion{I}{n}{x}e$ becomes $I$ under this isomorphism.
Let $I = \genby{\rho_1,\ldots,\rho_m}$ where we may assume that each relation $\rho_i$ consists of linear combinations of paths starting in the same vertex and ending in a single other one. Then by construction
 $\lineinsertion{I}{n}{x} = \genby{\lineinsertion{\rho_1}{n}{x},\ldots, \lineinsertion{\rho_m}{n}{x}}$.
We assume that these relations are ordered in such a way that $\lineinsertion{\rho_{l+1}}{n}{x}$, \ldots, $\lineinsertion{\rho_m}{n}{x}$ end in $x_0$ and all other relations $\lineinsertion{\rho_i}{n}{x}$ end in vertices different from $x_0$.
One can check that
\begin{align*}
    e\lineinsertion{I}{n}{x}e 
 &= e\genby{\lineinsertion{\rho_1}{n}{x}, \ldots, \lineinsertion{\rho_m}{n}{x}, \xi\lineinsertion{\rho_{l+1}}{n}{x},\ldots, \xi \lineinsertion{\rho_m}{n}{x} }e \\
 &= \genby{e \lineinsertion{\rho_1}{n}{x} e, \cdots, e \lineinsertion{\rho_l}{n}{x} e, e \xi \lineinsertion{\rho_{l+1}}{n}{x} e, \cdots, e \xi \lineinsertion{\rho_m}{n}{x} e}.
\end{align*}
To see the first equality, note that any path in $\lineinsertion{I}{n}{x}$ ending in $x_0$ will become zero in $e\lineinsertion{I}{n}{x}e$, unless it is extended (at least) by $\xi = x_0 \to \cdots \to x_n$.

Applying the isomorphism $\varphi$ from above gives $\varphi(e\lineinsertion{\rho_i}{n}{x}e)=\rho_i$ and $\varphi(e\xi \lineinsertion{\rho_j}{n}{x}e)=\rho_j$, where $1 \leq i \leq l$ and $l+1 \leq j \leq n$. So $\varphi(e\lineinsertion{I}{n}{x}e)=I$, completing the proof.

(2) 
We will apply \hyref{Proposition}{prop:recollement} here, so we first check the condition
 $\Ext^k_{\lineinsertion{\Lambda}{n}{x}}(\lineinsertion{\Lambda}{n}{x}/e, \lineinsertion{\Lambda}{n}{x}/e) = 0$ for $k\neq 0$.
We decompose $\lineinsertion{\Lambda}{n}{x}/e$ into a direct sum of indecomposable $\lineinsertion{\Lambda}{n}{x}/e$-modules. As left $\lineinsertion{\Lambda}{n}{x}$-modules these summands have projective resolutions $0 \to P(x_n) \to P(x_j)$, where $x_j \in Q_0(\lineinsertion{\Lambda}{n}{x}/e)$. In particular, the left $\lineinsertion{\Lambda}{n}{x}$-module $\lineinsertion{\Lambda}{n}{x}/e$ has projective dimension one.
So there can be at most a non-trivial $\Ext^1$, so we check for $x_j$ and another vertex $x_{j'}$:
\[ \xymatrix@R=3ex@C=3ex{
& 0 \ar[r] & P(x_n) \ar[r] \ar[d] \ar@{-->}[ld]_{\lambda \id} & P(x_j) \ar[r] \ar@{-->}[ld]^{0} & 0\\
0 \ar[r] & P(x_n) \ar[r] & P(x_{j'}) \ar[r] & 0
} \]
There is only one map $P(x_n) \to P(x_{j'})$ up to scalars; it is given by the path from $x_{j'}$ to $x_n$.
Therefore, there is a $\lambda \in \kk$ such that the triangle on the left commutes.
This extends to a null-homotopy showing that $\Ext^1$ indeed vanishes. Since $\projdim (\lineinsertion{\Lambda}{n}{x}e)_{\Lambda} < \infty$ by assumption, we can apply \hyref{Proposition}{prop:recollement} to complete the proof.

(3) and (4) follow from \hyref{Proposition}{prop:recoll-finite-gldim} and \hyref{Lemma}{lem:exseq}, respectively, where for (3), we use that $\lineinsertion{\Lambda}{n}{x}/e \cong \kk\A{n}$ has global dimension 1.
\end{proof}

\begin{theorem} \label{thm:insertion-sphericals}
Let $\Lambda$ be an algebra of finite global dimension and $\lineinsertion{\Lambda}{n}{x}$ be the $\A{n}$-insertion in $\Lambda$ at $x\in Q_0(\Lambda)$.
Let $F\in\Db(\Lambda)$ be a $d$-spherical object.

Then $\jincl F\in\Db(\lineinsertion{\Lambda}{n}{x})$ is a properly $d$-spherelike object if and only if $\Hom_{\Lambda}^\bullet(S(x),F)\neq0$, and in this case the spherical subcategory of $\jincl F$ is
\[  
\Db(\lineinsertion{\Lambda}{n}{x})_{\jincl F} \cong \Db(\Lambda) \oplus \Db(\kk\A{n-1}) .  
\] 
\end{theorem}

\begin{proof}
By \hyref{Lemma}{lem:insertion-basics} there is the recollement
\[
\recoll{\sod{S(x_0),\ldots,S(x_{n-1})}}
       {\Db(\lineinsertion{\Lambda}{n}{x})}
       {\Db(\Lambda)} .
\]
Note that $\Hom^\bullet_\Lambda(F,\pi\Naka S(x_i)) = \Hom^\bullet_\Lambda(F,\pi S(x_{i+1})[1])$ for $i=0,\ldots,n-1$,
since $\Naka S(x_i) = S(x_{i+1})[1]$ where $\Naka$ is the Nakayama functor for $\lineinsertion{\Lambda}{n}{x}$.
For $i \leq n-2$, we get $\pi S(x_{i+1}) =0$ by support reasons, which in turn implies $\Hom^\bullet_\Lambda(F,\pi\Naka S(x_i)) = 0$.
Hence, we can apply \hyref{Proposition}{prop:idempotent-sod} and get that $\jincl F$ is properly $d$-spherelike if and only if 
\[
\begin{split}
      \Hom^\bullet_{\lineinsertion{\Lambda}{n}{x}}(S(x_{n-1}), \jincl F) 
 &  = \Hom^\bullet_\Lambda(F, \pi S(x_n)[1])^*
    = \Hom^\bullet_\Lambda(F,S(x)[1])^* \\
 &  = \Hom^\bullet_\Lambda(\Naka^{-1}S(x),F[-1]) = \Hom^\bullet_\Lambda(S(x),F[d-1])
\end{split}
\]
is non-zero, as claimed.
Here we used $\pi S(x_n) = S(x)$ and, in the last step, the Nakayama functor for $\Lambda$ and $d$-sphericality of $F$.

Moreover, we deduce that there is a weak recollement
\[
\wrecoll{\sod{ S(x_0),\ldots,S(x_{n-2})}}{\Db(\lineinsertion{\Lambda}{n}{x})_{\jincl F}}{\Db(\Lambda)} .
\]
Using Serre duality as above, we have
\[ \Hom^\bullet_{\lineinsertion{\Lambda}{n}{x}}(S(x_i),\jincl M) =
   \Hom^\bullet_\Lambda(M,\pi S(x_{i+1})[1])^* , \]
and these spaces vanish for any $M \in \Db(\Lambda)$ and $0 \leq i \leq n-2$.
Therefore, the outer parts of the recollement are fully orthogonal, i.e.\
 $\Db(\lineinsertion{\Lambda}{n}{x})_{\jincl F} = \sod{S(x_0),\ldots,S(x_{n-2})} \oplus \jincl(\Db(\Lambda))$.

It remains to show that 
 $\mathcal{S} \coloneqq  \sod{S(x_0),\ldots,S(x_{n-2})} \cong \Db(\kk\A{n-1})$.
Let $M_j \coloneqq  M(j,n-2)$ be the unique indecomposable $\lineinsertion{\Lambda}{n}{x}$-module with top $S(x_j)$ and socle $S(x_{n-2})$.
Then $T=\bigoplus_{j=0}^{n-2} M_j$ is a tilting object in $\mathcal{S}$.
To see this, first note that any such $M_j$ has a projective resolution $P(x_{r_n})\to P(x_j)$.
Therefore, the projective dimension of $T$ is at most one. One can check that all $\Ext^1(M_k,M_l)=0$, so all $\Ext$-groups of $T$ vanish.
Since $0 \to M_{j+1} \to M_j \to S(x_j) \to 0$ is exact, $T$ is a tilting object.
One can calculate that the endomorphism ring of $T$ is indeed $\kk\A{n-1}$.
This completes the proof.
\end{proof}

\begin{remark} \label{rem:insertion-sphericals-sod}
\SOD\ holds in the setup of \hyref{Theorem}{thm:insertion-sphericals}, which can be seen in the following way.
The proof shows that for properly spherelike $\jincl F$,
\begin{equation} \label{eq:recollment-insertion}
  \recoll{ \sod{S(x_0),\ldots,S(x_{n-2})} }
         { \Db(\lineinsertion{\Lambda}{n}{x})_{\jincl F} }
         { \Db(\Lambda) }  ,
\end{equation}
is actually a recollement.
Now by \hyref{Lemma}{lem:insertion-basics}(4), the derived category $\Db(\lineinsertion{\Lambda}{n}{x}/e)$ has a full exceptional sequence
  $\sod{S(x_0),\ldots,S(x_{n-2}),S(x_{n-1})}$.
After $n-1$ successive left mutations to $S(x_{n-1})$, we arrive at another exceptional sequence
  $\sod{ \mathsf{L}^{n-1}(S(x_{n-1})), S(x_0),\ldots,S(x_{n-2}) }$
for $\Db(\lineinsertion{\Lambda}{n}{x}/e)$.
Combined with \hyref{Lemma}{lem:insertion-basics}(2), we obtain the recollement for $\Db(\lineinsertion{\Lambda}{n}{x})$:
\[ \recoll{ \sod{ \mathsf{L}^{n-1}(S(x_{n-1})), S(x_0),\ldots,S(x_{n-2}) } }
          { \Db(\lineinsertion{\Lambda}{n}{x}) }
          { \Db(\Lambda) }. \]
Comparing this recollement with the one from \eqref{eq:recollment-insertion}, we find
\[ \recoll{ \sod{\mathsf{L}^{n-1}(S(x_{n-1}))} }
          { \Db(\lineinsertion{\Lambda}{n}{x}) }
          { \Db(\lineinsertion{\Lambda}{n}{x})_{\jincl F} }   \]
so condition \SOD\ is met.
See \cite[\S 1]{Rudakov} or \cite[\S 2.3]{Bridgeland-Stern} for mutations.
Actually, we can compute $\mathsf{L}^{n-1}(S(x_{n-1}))$ explicitly. We have
  $\sod{S(x_0),\ldots,S(x_{n-1})} = \Db(\kk\A{n})$
by \hyref{Lemma}{lem:exseq}, and helix theory (see e.g.\ \cite[Cor.~2.10]{Bridgeland-Stern}) yields
\[ \mathsf{L}_{S(x_{n-2}),\ldots,S(x_0)} S(x_{n-1}) = \Naka_{\kk\A{n}} S(x_{n-1}) = I(x_{n-1}) . \]
Here we use that the simple on the sink $x_{n-1}$ of $\kk\A{n}$ is also the projective module of that vertex; so its Serre dual is given by the injective of the sink.
\end{remark}

\begin{remark} \label{rem:insertion_simultan}
\hyref{Theorem}{thm:insertion-sphericals} can be extended to iterated insertions at a subset $\Gamma\subset Q_0(\Lambda)$ with $\Hom^\bullet_\Lambda(S(x),F)\neq0$ for all $x\in\Gamma$.
The spherical subcategory behaves as expected: it is of the form $\Db(\Lambda) \oplus \bigoplus_{x \in \Gamma} \Db(\kk A_{n_x-1})$.

We treat $\Gamma=\{x,y\}$ and leave the general case to the reader. We denote $\Lambda' = \lineinsertion{\Lambda}{n}{x}$ and $\Lambda'' = \lineinsertion{\Lambda'}{m}{y}$ with $\jincl' \colon \Db(\Lambda') \into \Db(\Lambda'')$.
By the considerations of the previous remark and \hyref{Lemma}{lem:insertion-basics}, we obtain the recollement
\[   \recoll{\sod{S(y_0),\ldots,S(y_{m-1}),\jincl' I(x_{n-1})}}
            {\Db(\Lambda'')}
            {\Db(\Lambda')_{\jincl F}}   \]
As in the proof of \hyref{Theorem}{thm:insertion-sphericals}, we use the duality-adjunction argument from \hyref{Proposition}{prop:idempotent-sod} to deduce $\Hom^\bullet(S(y_j),\jincl'\jincl F) \iff j<m-1$. Moreover,
 $\Hom^\bullet(\jincl'I(x_{n-1}),\jincl'\jincl F) = \Hom^\bullet(I(x_{n-1}),\jincl F) \neq 0$,
as follows from the short exact sequence
 $0 \to S(x_{n-1}) \to I(x_{n-1}) \to C \to 0$ 
with $C \in \sod{S(x_0),\ldots,S(x_{n-2})}$ and therefore $C\in {}\orth\jincl F$.

Therefore we arrive at the recollement
\[
\recoll{\sod{S(y_0),\ldots,S(y_{m-2})}}
       {\Db(\Lambda'')_{\jincl'\jincl F}}
       {\Db(\Lambda')_{\jincl F}}
\]
which turns out to be
$\Db(\Lambda'')_{\jincl'\jincl F} \cong \Db(\Lambda) \oplus \Db(\kk A_{n-1}) \oplus \Db(\kk A_{m-1})$.
\end{remark}

\subsection{Tacking on quivers}

We start with the following data:
\begin{itemize}
\item $\Lambda=\kk Q/I$, a finite-dimensional algebra, given as a quiver $Q$ bound by an ideal $I$;
\item $(T,t)$, a finite quiver $T$ without oriented cycles and $t$ a sink in $T$;
\item $n\colon Q_0\to\IN$, assigning a multiplicity to every vertex of $Q$.
\end{itemize}
From this, we construct the following upper triangular matrix algebra 
\begin{align*}
(T,t)\tack_n\Lambda \coloneqq \begin{pmatrix} \Lambda & M \\ 0 & \kk T \end{pmatrix}
\quad \text{with} \quad
 M \coloneqq \bigoplus_{x\in Q_0(\Lambda)} (\Lambda e_x)^{\oplus n(x)} \otimes_\kk e_t \kk T
. \end{align*} 

In particular, this construction only depends on the algebra $\Lambda$, not on the bound quiver presentation $\kk Q/I$.
As quiver with relations $(T,t)\tack_n\Lambda$ is given by
\begin{itemize}
\item the vertex set $Q_0 \cup T_0$ (disjoint union);
\item the arrow set $Q_1 \cup T_1$ together with $n(x)$ arrows $t\to x$ for all $x\in Q_0$;
\item the relations $I$.
\end{itemize}
When $t$ and $n$ are understood, we simply write $T\tack\Lambda$ instead of $(T,t)\tack_n\Lambda$.

\begin{example}
For $T=A_1$ and any multiplicity function $n\colon Q_0\to\IN$, the construction coincides with the one-point extension on a projective module. This is given by a matrix algebra
\[ \left( \begin{array}{cc} \Lambda & M \\ 0 & \kk \end{array} \right)
   \qquad\text{with } M = \bigoplus_{x\in Q_0(\Lambda)} (\Lambda e_x)^{\oplus n(x)} . \]
\end{example}

\begin{example}
The derived-discrete algebra $\Lambda(r,n,m)$ of \hyref{Section}{sec:DDA} is obtained from tacking $\A{m}$ to the quiver of $\Lambda(r,n,0)$. Here the multiplicity function is 1 for the vertex of $\Lambda(r,n,0)$ which becomes the trivalent vertex of $\Lambda(r,n,m)$, and 0 else.
\end{example}

In the following lemma, we consider a particular idempotent
 $e \in T\tack\Lambda$, and the fully faithful functor it induces:
\[ e \coloneqq e_{Q_0} = \sum_{x \in Q_0} e_x, \qquad
   \jincl \coloneqq (T\tack\Lambda)e\otimes_\Lambda\blank \colon \Db(\Lambda) \to \Db(T\tack\Lambda)
. \]

It is well-known that $(1)-(3)$ of the lemma below actually hold for upper triangular matrix algebras
\[ \begin{pmatrix} \Lambda & M \\ 0 & B \end{pmatrix} \]
where $B$ has finite global dimension; see for example \cite[Prop.~3.3]{Li}.

\begin{lemma} \label{lem:tacking-basics}
Let $\Lambda$ be as above, $(T,t)$ and $n\colon Q_0(\Lambda)\to\IN$ as above.
Then
\begin{enumerate}
\item $\Lambda \cong e(T\tack\Lambda)e$ and $(T\tack\Lambda)/e \cong \kk T$;
\item $e$ is a recollant idempotent and so there is a recollement \[\recoll{\Db((T\tack\Lambda)/e)}{\Db(T\tack\Lambda)}{\Db(\Lambda)};\]
\item if $\Lambda$ has finite global dimension, then so does $T\tack\Lambda$;
\item the functor $\iincl = \Hom_{T\tack\Lambda}((T\tack\Lambda)/e,\blank) \colon \Db((T\tack\Lambda)/e) \into \Db(T\tack\Lambda)$ is fully faithful with image $\sod{S(y) \mid y\in T_0}$. Here these $S(y)$ form an exceptional sequence of $\Db(\kk T)$, and we consider them as objects of
 $\Db(T\tack\Lambda)$ via $T\tack\Lambda \to (T\tack\Lambda)/e = \kk T$.
Each $S(y)$ is an exceptional object of $\Db(T\tack\Lambda)$.
\end{enumerate}
\end{lemma}

\begin{proof}
(1) The isomorphisms $(T\tack\Lambda)/e \cong \kk T$ and $\Lambda \cong e(T\tack\Lambda)e$ are obvious from the construction of $T\tack\Lambda$ and the choice of $e=e_{Q_0}$, as there are no paths whatsoever from $Q$ to $T$.

(2) We are going to check the conditions of \hyref{Proposition}{prop:recollement}.
First, we show $\Ext^k_{T\tack\Lambda}((T\tack\Lambda)/e,(T\tack\Lambda)/e)=0$ for $k\neq 0$.
We can resolve $(T\tack\Lambda)/e$ as a $(T\tack\Lambda)/e$-module using direct sums of the indecomposable injective modules $I(j)$ for $j \in T_0$. 
Note that these modules stay injective when considered as $T\tack\Lambda$-modules, so the resolution stays the same.
Hence there is no $\Ext^1$ between those $T\tack\Lambda$-modules, because there is no $\Ext^1$ between the corresponding $(T\tack\Lambda)/e$-modules.

Next, we check that $(T\tack\Lambda)/e$ has finite projective dimension as a left $T\tack\Lambda$-module.
For this, we decompose $(T\tack\Lambda)/e$ into a direct sum of indecomposable projective $(T\tack\Lambda)/e$-modules. A projective resolution for one single summand as a $T\tack\Lambda$-module then looks like
\[ 0 \to \bigoplus_{t\to x} P(x)^{n(x)\cdot m(y)} \to P(y) \] 
where the sum is over the successors of $t$, so there are $n(x)$ arrows $t\to x$, and $m(y)$ is the number of paths from $y \in T_0$ to $t$.
This shows that the projective dimension of $(T\tack\Lambda)/e$ as a left $T\tack\Lambda$-module is at most $1$.

Finally, for any $y\in T = Q_0((T\tack\Lambda)/e)$, the injective resolution of the simple $S(y)$ as a $T\tack\Lambda$-module coincides with its $\kk T$-resolution, hence $\injdim_{T\tack\Lambda}(S(y))\leq1$. Therefore, all $(T\tack\Lambda)/e$-modules have finite injective dimension, and in particular the dual $D((T\tack\Lambda)/e)$ has. It follows that $(T\tack\Lambda)/e$ has finite projective dimension as a right $T\tack\Lambda$-module.

(3) and (4) follow from \hyref{Proposition}{prop:recoll-finite-gldim} and \hyref{Lemma}{lem:exseq}, respectively.
\end{proof}

\begin{theorem} \label{thm:tacking-sphericals}
Let $\Lambda$ be a finite-dimensional algebra of finite global dimension, $(T,t)$ a finite quiver with a sink and $n\colon Q_0(\Lambda)\to\IN$ any function, and let $T\tack\Lambda=(T,t)\tack_n \Lambda$ be the algebra obtained from tacking $T$ onto $\Lambda$ via $n$. Suppose $F\in\Db(\Lambda)$ is a $d$-spherical object.

Then $\jincl F\in\Db(T\tack\Lambda)$ is properly $d$-spherelike if and only if $n(x)>0$ for some $x\in\supp(F)$, and then the spherical subcategory is
\[  \Db(T\tack\Lambda)_{\jincl F} \cong \Db(\Lambda) \oplus \Db(\kk T/e_t)  . \]
\end{theorem}

\begin{proof}
By \hyref{Lemma}{lem:tacking-basics}, there is a recollement
\[
\recoll{\genby{S(y) \mid y \in T_0}}
       {\Db(T\tack\Lambda)}
       {\Db(\Lambda)}
\]
In order to apply \hyref{Proposition}{prop:idempotent-sod}, we look at
\begin{equation}
\label{eq:adj2}
    \Hom^\bullet(S(y),\jincl F)
  = \Hom^\bullet(\jincl F, \Naka S(y))^* 
  = \Hom^\bullet(F, \pi \Naka S(y))^*
\end{equation}
where $\pi = \Hom((T\tack\Lambda)e,\blank)$ and $y \in T_0$.
Fix a vertex $y\in T_0, y \neq t$. Then the simple $S(y)$ has a projective resolution $0\to P \to P(y) \to S(y) \to0$, where $P$ is the direct sum of $P(y')$ for all arrows $y \to y'$ in $T$. Hence $\Naka S(y)$ is isomorphic to $I \to I(y)$.
Kernel and cokernel of this map are supported on $T_0$ which follows from the construction of $T\tack\Lambda$, and as $y \neq t$.
Therefore $\pi \Naka S(y)=0$ and thus $S(y) \in {}\orth\jincl F$ for all $y \in T_0 \setminus \{t\}$ by \eqref{eq:adj2}.

Hence \hyref{Proposition}{prop:idempotent-sod} yields that $\jincl F$ is properly $d$-spherelike if and only if $\Hom^\bullet(S(t),\jincl F) \neq 0$.
Next we want to see that this non-vanishing is implied by $n(x)>0$ for some $x\in\supp(F)$.
The support condition translates to $e_x(H^j(F))\neq0$ for some $j\in\IZ$, hence $e_x(H^j(\jincl F))\neq0$. Let
 $0 \to P'\oplus P(x) \xrightarrow{a} P(t) \to S(t) \to 0$ be a projective resolution of $S(t)$ as an $T\tack\Lambda$-module; here $P(x)$ occurs as a non-trivial summand due to $n(x)>0$. Consider the morphism of complexes
\[ \xymatrix{
              & 0 \ar[d]^0 \ar[r]          & P'\oplus P(x) \ar[r]^a \ar[d]^{(0,\varphi)} & P(t) \ar[r] \ar[d]^0  & 0 \ar[d]^0 \\
\cdots \ar[r] & F^{j-1}     \ar[r]^{d^{j-1}} & F^j \ar[r]^{d^j}                        & F^{j+1} \ar[r]^{d^{j+1}} & F^{j+2} \ar[r]^{d^{j+2}}  & \cdots 
} \]
where $\varphi\colon P(x)\to\ker(d^j)\subset F^j$ is such that the induced map $P(x)\to H^j(F)$ is non-zero --- this is possible due to $e_x(H^j(F))\neq0$. In particular, $\varphi$ does not factor through $d^{j-1}$. Furthermore we have $\Hom(P(t),F^j)=0$, as $F^j$ is supported off $t$. Hence the morphism of complexes $(0,\varphi)$ is not null-homotopic and thus $\Hom^\bullet(S(t),\jincl F) \neq 0$.

Again by \hyref{Proposition}{prop:idempotent-sod}, we obtain a weak recollement
\[
\wrecoll{\sod{S(y)_{y\in T_0\setminus\{t\}}}}
        {\Db(T\tack\Lambda)_{\jincl F}}
        {\Db(\Lambda)}
\]
Replacing $F$ in \eqref{eq:adj2} by an arbitrary $M \in \Db(\Lambda)$ shows that the outer parts are mutually orthogonal, hence $\Db(T\tack\Lambda)_{\jincl F} \cong \Db(\Lambda) \oplus \sod{S(y)_{y\in T_0\setminus\{t\}}}$. 
The right-hand summand is equivalent to $\Db(\kk T/e_t)$ because of $\Db(\kk T) = \sod{S(y)_{y\in T_0}}$, and $t$ being a sink (so that the simple $S(t)$ comes last in the exceptional sequence).
\end{proof}

\begin{remark}
Following the argument of \hyref{Remark}{rem:insertion-sphericals-sod}, condition \SOD\ is again met.
\end{remark}

\begin{remark}
It would be interesting to extend \hyref{Theorem}{thm:tacking-sphericals} to more general upper triangular matrix algebras.
\end{remark}

\begin{remark} \label{rem:algebraic-blowup}
As Lemmas~\ref{lem:insertion-basics} and \ref{lem:tacking-basics} show, both constructions presented in this section share the following formal properties:
\begin{itemize}
\item They produce a new algebra from an algebra and combinatorial data.
\item The derived category of the new algebra is a recollement of the derived category of the old algebra and an exceptional sequence.
\end{itemize}
These properties are reminiscent of blowing up smooth, projective varieties in points; see \cite[\S 11.2]{Huybrechts}. The analogy is a bit closer: let $\Lambda$ be the algebra given by the Beilinson quiver
 $\xymatrix@C=1.4em@1{ 0 \ar@3[r] & 1 \ar@3[r] & 2 }$
(with commutativity relations), then $\Db(\Lambda)\cong\Db(\IP^2)$, by Beilinson's famous equivalence. This corresponds to the exceptional sequences $\sod{P(0), P(1), P(2)}$ and $\sod{\OO,\OO(1),\OO(2)}$ for $\Lambda$ and $\IP^2$, respectively. Blowing up $\IP^2$ in a point produces a strong exceptional sequence $\sod{\OO_E(-1)[-1],\OO,\OO(1),\OO(2)}$, which corresponds exactly to $A_1\tack_n\Lambda$ with $n(0)=1$ and $n(1)=n(2)=0$.
\end{remark}

\begin{lemma} \label{lem:arbitrary_posets}
Let $P$ be a finite poset. Then there exists a finite-dimensional hereditary algebra $\Lambda$ such that $P\subseteq \poset_1(\Lambda)$ as a subposet.
\end{lemma}

\begin{proof}
Write $P=(\{1, \ldots, m\}, \leq)$ for the finite poset $P$.
Let $S=2^P$ the set of all subsets of $P$. Then $S=(S, \subseteq)$ is a poset and the map
\[ \iota\colon P \to S, \quad p \mapsto \iota(p) \coloneqq \{ q \in P \mid q \leq p \} \]
is an inclusion of posets. We identify $P$ with its image in $S$.

 We consider the decomposable algebra
 $\Delta \coloneqq \kk Q^1 \times \cdots \times \kk Q^m$
with $m$ factors, where each $Q^i$ is a copy of the Kronecker quiver. We denote the sink and the source of $Q^i$ by $i'$ and $i''$, respectively. 
For any point $[\lambda:\mu] \in \IP^1$, there is a quasi-simple regular $Q^i$-representation $\xymatrix{\kk \ar@<0.6ex>[r]|{\,\lambda\,}  \ar@<-0.6ex>[r]|{\,\mu\,} & \kk }$. As a $\Delta$-module, it is a 1-spherical object of $\Db(\Delta)$. For each component $Q^i$, choose such spherical objects $F^i\in\Db(\Delta)$.

We will repeatedly tack quivers onto $\Delta$, and thereby obtain an algebra $\Lambda=\Delta_m$ with a spherelike poset containing $P$.

The construction is iterative: put $\Delta_0\coloneqq\Delta$. For $\Delta_{i-1}$ already defined, set 
\[ \begin{array}{l @{\hspace{1em}} l}
 \Delta_i \coloneqq A_1 \tack_{n_i} \Delta_{i-1} & \text{with $Q_0(A_1) \coloneqq \{i\}$ and} \\[1ex]
 n_i\colon Q_0(\Delta_{i-1}) \to \IN,              &
n_i(x) = 
\begin{cases}
1 & \text{if $x=j'$ and $i \notin \iota(j)$;}\\
0 & \text{else.}
\end{cases}
\end{array} \]
\hyref{Theorem}{thm:tacking-sphericals} shows that the spherical object $F^j\in\Db(\Delta)$ gives rise to a spherelike object $\bar F^j\in\Db(\Delta_m)$ with spherical subcategory
\[
\Db(\Delta_m)_{\bar F^j} =\mathsf{thick}(S(1'),S(1''),\ldots,S(m'),S(m''), \{ S(i) \text{  for } i \in \iota(j)\})
\]
In particular, 
\[      \Db(\Delta_m)_{\bar F^i} \subsetneq \Db(\Delta_m)_{\bar F^j} 
   \iff \iota(i) \subsetneq \iota(j) 
   \iff i < j, \]
so the claim holds for $\Lambda=\Delta_m$.
\end{proof}

This result has the following immediate consequence:

\begin{corollary} \label{cor:increasing_heights}
There is a sequence of hereditary algebras such that the heights (respectively widths) of their spherelike posets are strictly increasing.
\end{corollary}

\begin{example}[Posets of increasing heights] \label{ex:kronecker-tacking}
For any $i\in\IN$, consider the poset $P_i = \{1<2<\cdots<i\}$. We give a concrete description for the algebras $\Delta_i$ constructed in  proof: set $\Delta_0 = \kk Q^1 \times \cdots \times \kk Q^i$ where the $Q^j$ are Kronecker quivers, and $\Delta_j = A_1 \tack \Delta_{j-1}$. Tacking multiplicities in each step are 1 on smaller Kronecker sinks, and 0 else. By the lemma, $P_i \subseteq \poset_1(\Delta_i)$.

Below we show the quivers for $i=3$.
The $A_1$ tacked on is indicated via the dashed arrows.

\medskip
\noindent
\parbox[t]{0.3\textwidth}{
$\Delta_1: \xymatrix@R=2ex@C=1.85em{
 3'' \ar@<0.5ex>[r] \ar@<-0.5ex>[r] & 3' &   \\
 2'' \ar@<0.5ex>[r] \ar@<-0.5ex>[r] & 2' &   \\
 1'' \ar@<0.5ex>[r] \ar@<-0.5ex>[r] & 1' & 1 
} $
}
\hfill
\parbox[t]{0.3\textwidth}{
$\Delta_2: \xymatrix@R=2ex@C=1.85em{
 3'' \ar@<0.5ex>[r] \ar@<-0.5ex>[r] & 3' &   \\
 2'' \ar@<0.5ex>[r] \ar@<-0.5ex>[r] & 2' & 2 \ar@{-->}[dl] \\
 1'' \ar@<0.5ex>[r] \ar@<-0.5ex>[r] & 1' & 1 
} $
}
\hfill
\parbox[t]{0.3\textwidth}{
$\Delta_3: \xymatrix@R=2ex@C=1.85em{
 3'' \ar@<0.5ex>[r] \ar@<-0.5ex>[r] & 3' & 3 \ar@{-->}[ld] \ar@{-->}[ldd] \\
 2'' \ar@<0.5ex>[r] \ar@<-0.5ex>[r] & 2' & 2 \ar[dl] \\
 1'' \ar@<0.5ex>[r] \ar@<-0.5ex>[r] & 1' & 1 
} $
}
\end{example}

\begin{remark}
Also in geometry, one can construct a poset of arbitrary height. For example, by iteratively blowing up a $(-2)$-curve on a smooth projective surface.
\end{remark}

\begin{example}[Poset containing a cycle] \label{ex:poset_cycle}
\hyref{Lemma}{lem:arbitrary_posets} gives a recipe to construct an algebra $\Delta$ whose spherical poset contains a cycle, i.e.\ the poset $P=\{1,2,3,4\}$ with predecessor sets
 $\iota(1) = \{1\}, \iota(2) = \{1,2\}, \iota(3) = \{1,3\}, \iota(4) = \{1,2,3,4\}$.
Following the proof of the lemma, we need four copies of the Kronecker quiver, to which four $A_1$-quivers are tacked:
\[ 
\parbox{0.3\textwidth}{
\qquad \xymatrix@=1.0ex{
  & 4 \ar@{-}[dl] \ar@{-}[dr] &   \\
  2 &                         & 3 \\
  & 1 \ar@{-}[ul] \ar@{-}[ur] &   
}

Hasse diagram of $P$
}
\qquad\qquad
\parbox{0.5\textwidth}{
\xymatrix@R=2ex{
 2'' \ar@<0.6ex>[r] \ar@<-0.6ex>[r] & 2'              & 1                         & 4 ' & 4'' \ar@<0.6ex>[l] \ar@<-0.6ex>[l] \\ 
                                    & 3 \ar[u] \ar[d] & 4 \ar[ul] \ar[dr] \ar[dl] & \\
 1'' \ar@<0.6ex>[r] \ar@<-0.6ex>[r] & 1'              & 2 \ar[l] \ar[r]           & 3 ' & 3'' \ar@<0.6ex>[l] \ar@<-0.6ex>[l] \\
} 
}
\]

\medskip\noindent
As is clear from inspecting the resulting quiver, the algorithm is not optimal. For example, the disconnected vertex $1$ is superfluous.
\end{example}

\subsection{Spherical objects from cluster-tilting theory} \label{sub:general}
We have seen two constructions that produce spherelike objects out of spherical ones. In this section, we use a general recipe of Keller \& Reiten \cite{Keller-Reiten} to obtain spherical objects. Their construction yields algebras $\Lambda$, which have perfect derived categories $\Kb(\projectives{\Lambda})$ containing $(d+1)$-Calabi--Yau objects. For $d=1$ and $d=2$, there are always $(d+1)$-spherical objects among these $(d+1)$-Calabi--Yau objects.

Let $\EE$ be a $\kk$-linear, idempotent complete Frobenius category, such that $\projectives{\EE}=\add(P)$ for some $P \in \EE$. Let
 $\CC=\underline{\EE} \coloneqq \EE/\projectives{\EE}$
be the associated stable category, which is triangulated by work of Happel \cite{Happel}. We assume that $\CC$ is Hom-finite and $d$-Calabi--Yau, i.e.\ $[d]$ is a Serre functor for $\CC$. An object $T \in \CC$ is called \emph{$d$-cluster-tilting object} if
\[  \add(T) = \{N\in\CC \mid \Ext^i_{\CC}(T,N)=0 \text{ for } 1\leq i\leq d-1\} . \] 
The \emph{higher Auslander algebra} $\Lambda$ and the \emph{cluster-tilted algebra} $\uLambda$ of $T$ are:
\[ \Lambda = \End_{\EE}(P \oplus T) \qquad\text{and}\qquad \uLambda=\End_\CC(T) . \]
See \cite{Iyama} for $\Lambda$, and note that $\uLambda$ is the quotient algebra of $\Lambda$ by the ideal of morphisms factoring through $\add(P)$.

By \cite[Thm.~5.4(c)]{Keller-Reiten}, every $\uLambda$-module has finite projective dimension as a $\Lambda$-module. Therefore, the subcategory $\DD^b_\smuLambda(\Lambda) \subseteq \Db(\Lambda)$ of complexes with cohomologies in $\modules{\uLambda}$ is contained in $\Kb(\projectives{\Lambda})$. This subcategory is equal to the thick subcategory
 $\mathsf{thick}(\modules{\uLambda}) \subseteq \Kb(\projectives{\Lambda})$
generated by finite-dimensional $\uLambda$-modules. In particular, this subcategory is Hom-finite.

Moreover, again by \cite[Thm.~5.4(c)]{Keller-Reiten}, every object of $\mathsf{thick}(\modules{\uLambda})$ is a $(d+1)$-Calabi--Yau object of $\Kb(\projectives{\Lambda})$. It seems to be unknown whether $\mathsf{thick}(\modules{\uLambda})$ is a $(d+1)$-Calabi--Yau category.

\begin{lemma}\label{lem:Keller-Reiten}
  With the notation introduced above, let $d=1$ or $d=2$. Let $F$ be a $\uLambda$-module satisfying $\End_\smuLambda(F, F)=\kk$ and $\Ext^1_\smuLambda(F, F)=0$.
\begin{enumerate}
\item  The object $F$ is $(d+1)$-spherical in
       $\mathsf{thick}(\modules{\uLambda}) \subseteq \Kb(\projectives{\Lambda})$.
\item If $\Lambda$ is finite-dimensional, then $F$ is $(d+1)$-spherical in $\Kb(\projectives{\Lambda})$.
\end{enumerate}
\end{lemma}

\begin{proof}
We already know that $F$ has the required Calabi--Yau property. It remains to show that $F$ is $(d+1)$-spherelike.

We have $\kk \cong \End(F) \cong \Hom(F, F[d+1])^{*}$ by our assumptions and the functorial CY-isomorphism: $\Hom_\Lambda(F,\blank) \isom \Hom_\Lambda(\blank,F[d+1])^*$. This isomorphism also implies that $F$ has projective dimension at most $d+1$ as a $\Lambda$-module. Hence $\Hom(F, F[n])=0$ for all $n<0$ and $n>d+1$. Since $\Modules{\uLambda} \subseteq \Modules{\Lambda}$ is closed under extensions, it is well-known that $\Ext^1_\smuLambda(X, Y)= \Ext^1_\Lambda(X, Y)$ for $\uLambda$-modules $X, Y$, see e.g.\ \cite[Lem.~3.2]{Oort}. By our assumption this completes the proof for $d=1$. For $d=2$, the claim follows from the CY isomorphism: $0=\Hom(F, F[1]) \cong \Hom(F, F[2])^{*}$.
\end{proof}

As a consequence we obtain the following.

\begin{proposition} \label{prop:cluster-sphericals}
Let $d=1$ or $d=2$ and assume that $\Lambda$ is finite-dimensional. Then every simple $\underline{\Lambda}$-module $S$ is $(d+1)$-spherical in $\Kb(\projectives{\Lambda})$.
\end{proposition}
\begin{proof}
We can apply \hyref{Lemma}{lem:Keller-Reiten}. Indeed $\End(S)=\kk$ by Schur's Lemma and $\Ext^1_\smuLambda(S, S)=\Ext^1_\Lambda(S, S)$ (see proof of \hyref{Lemma}{lem:Keller-Reiten}). Since $\Lambda$ is finite-dimensional and $S$ has finite projective dimension as a $\Lambda$-module, the latter Ext-group vanishes by the validity of the strong no loops conjecture, see \cite{Igusa-Liu-Paquette}. This completes the proof.
\end{proof}

\begin{remark} \label{rem:no-loops}
Here is a more general criterion for arbitrary $d>0$: let $T'$ be an indecomposable direct summand of $T$  and let $S'$ be the simple $\Lambda$-module corresponding to the projective module $P' \coloneqq \Hom_\Lambda(P\oplus T, T')$. Then the following statements are equivalent:
\begin{enumerate}
\item $S'$ is a $(d+1)$-spherical $\Lambda$-module;
\item \label{it:loop-summand} 
      in the $d$-AR sequence starting and ending at $T'$,
      none of the inner terms contain $T'$ as a direct summand.
\end{enumerate}
The existence of a $d$-AR sequence starting and ending in $T'$ is the CY-condition, whereas the absence of $T'$ as an inner term means that \emph{$T$ has no loops at $T'$}; see \cite[Def.~5.4]{Iyama-Yoshino}.
To see the equivalence of these statements, apply $\Hom_\Lambda(P\oplus T,\blank)$ to the AR sequence starting and ending in $T'$, yielding a projective resolution of $S'$. 
\end{remark}

\begin{example}[Auslander algebra] \label{ex:Auslander}
Let $R\coloneqq \kk[x]/(x^3)$. By definition, the Auslander algebra of $\EE=\modules{R}$ is $\Lambda\coloneqq \End_R(R\oplus\kk[x]/(x^2)\oplus\kk[x]/(x))$. It is a higher Auslander algebra with quiver
\[
\xymatrix@C=1.3em{
          1 \ar@<0.5ex>[r]^a &
          2 \ar@<0.5ex>[r]^b \ar@<0.5ex>[l]^c & 
          3 \ar@<0.5ex>[l]^d }
\]
and relations $c a=0$ and $ac-db=0$. Here, the vertices $1$, $2$ and $3$ correspond to the $R$-modules $\kk[x]/(x)$, $\kk[x]/(x^2)$ and $R$. Moreover, $a$ and $b$ are inclusions whereas $c$ and $d$ are projections.
One can check that the simple $\Lambda$-modules $S(1)$ and $S(2)$ are $2$-spherical objects in $\Db(\Lambda)$, in accordance with \hyref{Proposition}{prop:cluster-sphericals} for $d=1$. This example generalises to $R = \kk[x]/(x^n)$.
 \end{example}
 
 \begin{example}[Preprojective algebra] \label{ex:RelCluster}
The preprojective algebra $\Pi\coloneqq\Pi(A_3)$ of type $A_3$ is given by the quiver
\[
\xymatrix@C=1.3em{
          1 \ar@<0.5ex>[r]^a &
          2 \ar@<0.5ex>[r]^b \ar@<0.5ex>[l]^c & 
          3 \ar@<0.5ex>[l]^d }
\]
with relations $ca = 0 = bd$ and $ac-db=0$. Then $\EE=\modules{\Pi}$ is a Frobenius category. The stable category $\stmodules{\Pi}$ is a $2$-Calabi--Yau triangulated category, known as the $2$-cluster category of type $A_3$, see \cite{BMRRT, Geiss-Leclerc-Schroer}. More generally, one could take the triangulated categories arising from preprojective algebras of acyclic quivers $Q$ and Weyl group elements $w \in W_{Q}$, see \cite{BIRSc, Geiss-Leclerc-Schroer}.
\[
 T = T_{1} \oplus T_{2} \oplus T_{3} \coloneqq
     1 \oplus \begin{array}{l}\scriptstyle ~2\\[-0.8ex] \scriptstyle 1 \end{array} 
       \oplus \begin{array}{l}\scriptstyle ~2\\[-0.8ex] \scriptstyle1~3 \end{array} 
 \]
is a $2$-cluster-tilting object in $\stmodules{\Pi}$.
We consider the corresponding higher Auslander algebra $\Lambda\coloneqq \End_\Pi(T\oplus \Pi)$ with quiver
\[ \xymatrix{
6 \ar[rr]^h &             & 5 \ar[dr]_e &                       & 4 \ar[dr]^b \ar[ll]_d & \\
            & 2 \ar[lu]^g &             & 3 \ar[ur]_a \ar[ll]^f &                       & 1 \ar[ll]^c \\
} \]
and relations $da-hgf, ed-cb, ac, ba, feh, ehg$.
Here, the vertices $1$ to $3$ correspond to the modules $T_1$ to $T_3$, the vertices $4$ to $6$ correspond to the projective $\Pi$-modules $P(1)$ to $P(3)$. Then the simple $\Lambda$-modules $S(1)$ to $S(3)$ are $3$-spherical objects, in accordance with \hyref{Proposition}{prop:cluster-sphericals} for $d=2$.
\end{example}

\section{Circular quivers}
\label{sec:circular-quivers}

\noindent
Let $Q_n$ be the clockwise oriented circle with $n$ vertices $1,2,\ldots,n$ and arrows $i\xrightarrow{a_i}i+1$, with $n+1=1$.
For any tuple of integers $n>1$, $t \geq 1$ and $r_1, \ldots, r_t$ with $1\leq r_1< \cdots < r_t < n$, we define the following `circular' $\kk$-algebra
\[  C_n(r_1,\ldots ,r_{t})\coloneqq \kk Q_n/(a_{r_2} \cdots  a_{r_1}, \ldots, a_{r_{t}} \cdots a_{{r_{t-1}}}, a_{n} \cdots  a_{r_{t}})  \]
These are finite-dimensional Nakayama algebras of global dimension $t+1$.
As a special case, we introduce the `circular basic algebra' $\CB_n$
\begin{align*}
   \CB_n &\coloneqq  C_n(1, 2, \ldots, n-1).
\end{align*}
In other words, $\CB_n$ has all but one possible zero relations of length two.
\begin{figure}
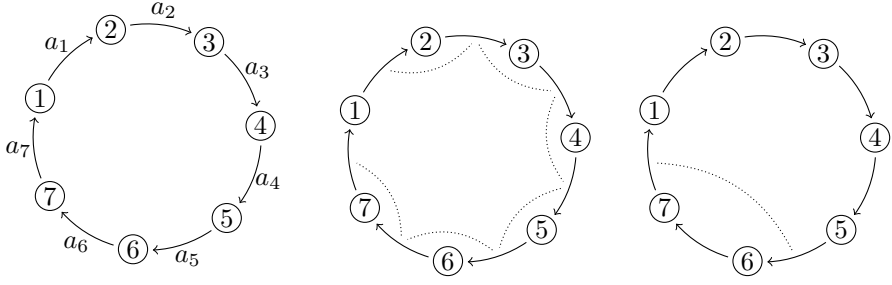

\tikzcircular
\caption{ \label{fig:quiver:Qn}
Left:   $Q_7$. \qquad
Centre: $\CB_7$. \qquad
Right:  $C_7(5)$.
}
\end{figure}

\begin{lemma} \label{lem:CB}
The simple module $S(1)$ is a $t$-spherical object in $\Db(\CB_{t})$.
\end{lemma}
\begin{proof}
The minimal projective resolution of $S(1)$ is given as follows 
\[  0 \to P(1) \to P(t) \to \cdots \to P(3) \to P(2) \to P(1) \to S(1) \to 0  . \]
Since $\Hom^\bullet(P(i),S(j)) = \kk \cdot \delta_{ij}$, the object $S(1)$ is $t$-spherelike. Applying the Nakayama functor to $S(1)$, we get a complex of injectives $I(1) \to \cdots \to I(3) \to I(2) \to I(1)$, which has only one non-zero cohomology: namely, $S(1)$ in degree $t$. Hence $S(1)$ is a $t$-spherical object.
\end{proof}

\begin{remark}
Let $\CI_{d}\coloneqq\kk Q_d/(a_2 a_1, a_3 a_2, \ldots, a_1 a_d)$, so that all possible zero relations of length two occur. This is a self-injective algebra, yielding a Frobenius category $\modules{\CI_d}$, which satisfies the conditions from \hyref{Section}{sub:general}. The stable category $\stmodules{\CI_d}$ is a generalized $d$-cluster category of type $A_1$, i.e.\ it is given as a triangulated orbit category $\Db(\kk)/[d]$, see \cite{Keller-orbit}. In particular, it is $d$-CY. Each of the $d$ indecomposable objects $X_i$ in this category is $d$-cluster-tilting and their \emph{relative cluster-tilted algebras} $\End_{\CI_d}(\CI_d \oplus X_i)$ are isomorphic to the algebra $\CB_{d+1}$ from above. This explains why the simple module $S(1)$ is a $(d+1)$-CY object in $\Db(\CB_{d+1})$ (see \hyref{Lemma}{lem:CB}).
\end{remark}

Note that any circular algebra $C=C_n(r_1,\ldots,r_t)$ is built from $\CB_{t+1}$ by simultaneous $\A{r_{i+1}-r_i-1}$-insertions at all vertices.
Then the idempotent $e=\sum_{i=1}^{t+1} e_{r_i}$ yields an isomorphism $\CB_{t+1} \cong eCe$.
By \hyref{Lemma}{lem:insertion-basics}, this induces a fully faithful embedding $\jincl\colon \Db(\CB_{t+1}) \to \Db(C_n(r_1,\ldots,r_t))$.

\begin{proposition} \label{prop:circular-df}
Let $C=C_{n}(r_{1}, \ldots, r_{t})$ be a circular algebra. Then $\jincl(S(1))$ is a $(t+1)$-spherelike object of $\Db(C)$ and its spherical subcategory is the following derived category (with $r_0\coloneqq 0$ and $r_{t+1}\coloneqq n$)
\[  \Db(C)_{\jincl(S(1))} \cong \Db( \CB_{t+1}) \oplus \bigoplus_{i=0}^t \Db(\kk\A{r_{i+1}-r_{i}-2})  . \]
\end{proposition}

\begin{proof}
This follows from \hyref{Theorem}{thm:insertion-sphericals} and \hyref{Remark}{rem:insertion_simultan}, it just remains to check that $\Hom^\bullet_{\CB_{t+1}}(S(i), S(1)) \neq 0$ for all $i=1, \ldots, t+1$. Indeed,
\[  0 \to P(1) \to P(t+1) \to \cdots \to P(i+2) \to P(i+1) \to P(i) \to S(i) \to 0  \]
is the minimal projective resolution of the simple $\CB_{t+1}$-module $S(i)$.
\end{proof}

\begin{example}
Consider the circular algebra $C\coloneqq C_7(5)=\kk Q_7/(a_7a_6a_5)$ from \hyref{Figure}{fig:quiver:Qn}. It is isomorphic to an $\A{4}$-insertion at vertex $1$ followed by an $\A{1}$-insertion at vertex $2$ for the algebra $\CB_2$. Set $e=e_5+e_7$. 
According to \hyref{Proposition}{prop:circular-df}, the complex of $C$-modules $\jincl(S(1))=P_5 \xrightarrow{\cdot a_4 a_3 a_2a_1a_7} P_7 \xrightarrow{\cdot a_6 a_5} P_5$ is $2$-spherelike and the spherical subcategory is given by
\[  \Db(C)_{\jincl(S(1))} \cong \Db(eCe) \oplus \sod{S(1), S(2), S(3)}
                   \cong \Db(\CB_2) \oplus \Db(\kk\A{3})  . \]
\end{example}

\section{Canonical algebras and weighted projective lines}
\label{sec:canonical_algebras}

\noindent
The derived categories of canonical algebras or, equivalently by tilting, weighted projective lines possess properly spherelike objects in the exceptional tubes. See \cite{Geigle-Lenzing} and \cite{Ringel} for a general reference to this class of algebras.

Fix a weight sequence $p=(p_1,\ldots,p_t)$ with all $p_i\geq2$ and a sequence of distinct points $\lambda=(\lambda_3,\ldots,\lambda_t)$ with $\lambda_i=(a_i:b_i)\in\smash{\IP^1_\kk}$. We denote by 
 $C(p_1,\ldots,p_t;\lambda_3,\ldots,\lambda_t) = C(p;\lambda) = \kk Q(p)/I$
the associated canonical algebra, where $Q(p)$ is the quiver consisting of a source 0, a sink 1 and $t$ paths $\vec{p}_i$ of length $p_i$ from 0 to 1. There are $t-2$ relations for $C(p;\lambda)$, given by $x_i^{p_i} = a_i x_2^{p_2} - b_i x_1^{p_1}$ for $3\leq i\leq t$, where all arrows along the path $\vec{p}_i$ are denoted $x_i$;
see \cite[Sec.\ 4]{Geigle-Lenzing}.
One can assume all $a_i=1$ and $b_i\in\kk$ and, moveover, $b_3=1$. 

The quiver for $C(5,4,6)$:

{\centering
\tikzcanonicalalgebra
\par }

There is a family of tubes in $\Db(C(p;\lambda))$ and among these, there are $t$ tubes of ranks $p_1,\ldots,p_t$ at the points $\lambda_1,\ldots,\lambda_t\in\IP^1$. See \hyref{Figure}{fig:canonical:tubes}. 
For each other point in $\IP^1$ there is a tube of rank one.

\begin{figure}
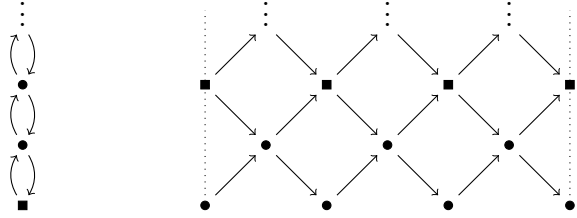

\tikztubes
\caption{  \label{fig:canonical:tubes}
AR-quivers of a homogoneous tube (left) and an exceptional tube of rank 3 (right) of $C(p,\lambda)$. The dotted edges of the tube are identified. The spherical and spherelike modules are drawn as boxes.}
\end{figure}

\begin{remark}
\label{rem:fracCY}
Let $M \in \Db(C(p;\lambda))$ be an object of a tube of rank $r$. Then $\tau^r M\cong M$ and hence the tube is $\frac{r}{r}$-fractionally Calabi--Yau, i.e.\ $\Naka^r M \cong M[r]$. Since tubes in derived categories of canonical algebras are standard, see \cite[Thm.~1.6, Cor.~1.7]{Skowronski}, we can compute $\Hom^\bullet(M,N)$ using the diagrams in \hyref{Figure}{fig:canonical:tubes} and get for $M$ indecomposable of quasi-length $s$:
\[  \dim \Hom(M,M)   = \left\lceil  \frac{s}{r} \right\rceil  ,  \qquad
    \dim \Ext^1(M,M) = \left\lfloor \frac{s}{r} \right\rfloor .  \]
So for a tube of rank $r$, all indecomposable modules of quasi-length $r$ are 1-spherelike objects, and they are mapped into each other by AR-translations. Moreover, these are spherical if and only if $r=1$.
\end{remark}

Explicitly, for the $i$-th tube of rank $p_i$, we obtain one of the spherelike objects, which we call $F_i$, as the cokernel of its projective resolution
 $P(1) \raisebox{-1pt}{$\xrightarrow{\cdot w}$} P(0)$
where $w=\vec{p}_i$ is the path from $0$ to $1$ along the $i$-th arm.
With this resolution one can check that we obtain indeed spherelike objects for each tube.
Finally, $F_i$ has the following representation: one-dimensional vector spaces on all vertices of $Q_0(p)$; the identity for all arrows of $Q_1(p)$, except for zero on the single arrow on the $i$-th arm going into the sink.

Note that all other spherelike objects in the $i$-th tube can be obtained from $F_i$ by AR-translations.

\begin{proposition}
\label{prop:sphsubcat-canonical}
For $p=(p_1,\ldots,p_t)$ and $\lambda=(\lambda_1,\ldots,\lambda_t)$, the $C(p,\lambda)$-module $F_i$ is a 1-spherelike object of $\Db(C(p;\lambda))$, with spherical subcategory
\[ \Db(C(p;\lambda))_{F_i} \cong \Db(C(p_1,\ldots \hat p_i \ldots,p_t; \lambda_1,\ldots \hat\lambda_i\ldots,\lambda_t)) \oplus \Db(\kk\A{p_i-2}) . \]
\end{proposition}

\begin{corollary}
Given $p=(p_1,\ldots, p_t)$ and $\lambda=(\lambda_1,\ldots,\lambda_t)$ as above, let $N$ the number of pairwise different weights $p_i$. Then the stable spherelike poset of $\Db(C(p;\lambda))$ contains the discrete poset of $N$ elements:
\[ \{\bullet_1\cdots\bullet_N\} \subseteq \posetaut(C(p,\lambda)) .\]
\end{corollary}

\begin{proof}[Proof of the proposition]
Without loss of generality we can assume $i=t$. We first look at the special case that $p_t = 2$. Denote $C=C(p_1,\ldots,p_{t-1},2;\lambda)$ and let $e$ the idempotent yielding $eCe = C(p_1,\ldots,p_{t-1};\lambda_1,\ldots,\lambda_{t-1})$. Using the relations in $C$, $w=\vec{p}_t$ is also a combination $L=a_t\vec{p}_1+b_t\vec{p}_2$ of paths in $eCe$, so we can write $F=F_t$ as $\jincl E$ for the $eCe$-module
 $E \coloneqq \coker(P_{eCe}(1) \smash{\raisebox{-1pt}{$\xrightarrow{\cdot L}$}} P_{eCe}(0))$,
where $\jincl\colon \CC\coloneqq \Db(eCe)\into\DD\coloneqq\Db(C)$. Since $E$ is a homogenous quasi-simple for $eCe$, it is $1$-spherical in $\CC$.

We have to show $ \DD_{F} \cong \CC$. From \hyref{Theorem}{thm:projection_functor} we get the weak recollement $\wrecoll{\CC\orth\cap{}\orth F}{\DD_F}{\CC}$. Thus we are left to show $\CC\orth\cap{}\orth F = 0$. As in \hyref{Lemma}{lem:insertion-basics}, $\CC\orth = \sod{S}$ with $S = C/e$. Using the representation of $F$, we see that $S$ is a submodule of $F$, so $\Hom(S,F) \cong \kk \neq 0$.

Now denote $\EE \coloneqq \Db(C(p;\lambda))$. Since $C(p;\lambda) = C(p_1,\ldots,p_t;\lambda_1,\ldots,\lambda_t)$ is obtained from $C$ by $\A{p_t-1}$-insertion, \hyref{Lemma}{lem:insertion-basics} gives the recollement $\recoll{\sod{S(1),\ldots,S(p_t-1),S}}{\EE}{\CC}$. Combining the previous calculation with \hyref{Theorem}{thm:projection_functor} yields the weak recollement
\[  \wrecoll{ \sod{S(1),\ldots,S(p_t-1)} \cap {}\orth F }{ \EE_F }{ \CC }  . \]
By arguments analogous to the ones given in the proof of \hyref{Theorem}{thm:insertion-sphericals}, we obtain the statement.
\end{proof}

\begin{remark}
\label{rem:DMcurves}
This statement can be interpreted geometrically, i.e.\ from the point of view of weighted projective lines or of Deligne--Mumford curves: while skyscraper sheaves in ordinary points are spherical, skyscraper sheaves of length 1 supported on points with non-trivial isotropy are exceptional. Moreover, if the local isotropy at a point $x$ has order $p$, then a $p$-fold extension of these exceptional sheaves is a 1-spherelike sheaf. The spherical subcategory of that sheaf contains the weighted/stacky curve with trivial isotropy at $x$.

We mention that the objects at the mouth of exceptional tubes form what is called an exceptional cycle in \cite{BPP}. There is an associated autoequivalence which in this case goes back to Meltzer's tubular mutations \cite{Meltzer}. From the geometric point of view, these are line bundle twists on the stacks.
\end{remark}

\section{Derived-discrete algebras} \label{sec:DDA}

\noindent
We turn to algebras with discrete derived categories, also called `derived-discrete algebras'. These have been introduced by Vossieck \cite{Vossieck} and also classified up to Morita equivalence. Bobi\'nski, Gei\ss\ \& Skowro\'nski \cite{BGS} provide normal forms for the derived equivalence classes of these algebras, and they also compute the Auslander--Reiten quivers of the derived categories.

As we are interested in spherelike objects and their spherical subcategories, we employ the derived normal form of \cite{BGS} and consider the algebras $\Lambda(r,n,m)$ whose bound quiver consists of an oriented cycle of length $n\geq2$ with $1\leq r<n$ consecutive zero relations and, for $m\geq0$, an oriented $\A{m}$-quiver inserted into the last vertex of the cycle part of the $r$ relations.
We stress that $n>r$ translates to finite global dimension of $\Lambda(r,n,m)$.

\begin{center}
\includegraphics[width = 0.6\textwidth]{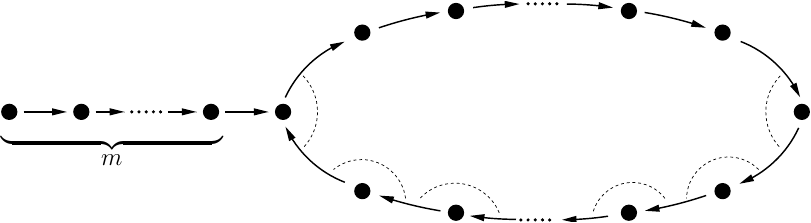}
\end{center}

By \cite{BGS}, the Auslander--Reiten quiver of $\Db(\Lambda(r,n,m))$ consists of $3r$ components $\XX^0,\ldots,\XX^{r-1}$, $\YY^0,\ldots,\YY^{r-1}$, $\ZZ^0,\ldots,\ZZ^{r-1}$ and by abuse of notation we denote the full subcategories of $\Db(\Lambda(r,n,m))$ of objects in an AR-component by the same symbol.

The $\XX$ and $\YY$ components are of type $\IZ A_\infty$, whereas the $\ZZ$ components are of type $\IZ A_\infty^\infty$. Hence there is a notion of height $h(A)$ for indecomposable objects $A$ in the $\XX$ and $\YY$ components; we take height 0 to mean objects at the mouths.
We write $\XX$ for the additive hulls of the union of $\XX^0,\ldots,\XX^{r-1}$, and analogously for $\YY$ and $\ZZ$. The subcategories $\XX$ and $\YY$ are triangulated, and the shift in these categories satisfies 
$X[r] = \tau^{-m-r}(X)$ and $Y[r] = \tau^{n-r}(Y)$ for $X \in \XX$ and $Y \in \YY$, the objects of $\XX$ and $\YY$ are fractionally Calabi--Yau of CY-dimensions $\frac{m}{m+r}$ and $\frac{n}{n-r}$, respectively; see \cite[Thm.~B]{BGS}.

The existence of spherelike objects is fully understood: \cite[Prop.~5.4]{BPP} states that all indecomposable objects in $\ZZ$ are exceptional, and moreover
\[ \begin{array}{rcl}
   X\in\ind\XX \text{ spherelike} & \iff & h(X) = m+r-1; \\
   Y\in\ind\YY \text{ spherelike} & \iff & h(Y) = n-r-1.
\end{array} \]
and then $X$ is $(1-r)$-spherelike, and $Y$ is $(1+r)$-spherelike, respectively.
In certain cases the spherelike objects are actually spherical:
\[ \begin{array}{rcl}
   X\in\ind\XX \text{ spherical} & \iff & m=0, r=1 \text{ and } h(X)=0 ; \\
   Y\in\ind\YY \text{ spherical} & \iff & n=r+1 \text{ and } h(Y)=0.
\end{array} \]

\begin{proposition} \label{prop:spherelikes_DDA}
Let $X\in\XX$ and $Y\in\YY$ be nilpotent properly spherelike objects in $\DD = \Db(\Lambda(r,n,m))$.
Then the spherical subcategories $\DD_X, \DD_Y$ are generated by exceptional sequences and
satisfy \SOT, \SOD. 
Moreover,
\begin{align*}
 \DD_Y &=  \Db(\Lambda(r,n,m))_Y \cong \Db(\kk\A{n-r-2}) \oplus \Db(\Lambda(r,r+1,m)) \text{, and} \\
 \DD_X &=  \Db(\Lambda(1,n,m))_X \cong \Db(\kk\A{m-1}) \oplus \Db(\Lambda(1,n,0)) \text{ if $r=1$}.
\end{align*}
If $r>1$, then the spherical subcategory $\DD_X$ is not the bounded derived category of a finite-dimensional $\kk$-algebra.
\end{proposition}

\begin{proof}
Fix an indecomposable object $X\in\XX$ of height $m+r-1$. Then $X$ is $(1-r)$-spherelike and
 $\omega(X) \coloneqq \nu X[r-1] = \tau X[r] = \tau^{1-m-r}X$,
using that $X[r] = \tau^{-m-r}(X)$. Up to scalars, there is a unique non-zero morphism $X\to\omega(X)$, and its cone is the asphericality $Q_X$. Denote by $E\in\XX$ the unique indecomposable object of height $m+r-2$ admitting an irreducible morphism $E\to X$. Note that $E$ is an exceptional object.

We claim that $Q_X \cong E[1]\oplus E[r]$. For this, we consider the triangle $E\to X\to X_0$, where $X_0\in\ind\XX$ is the unique object with $h(X_0)=0$ and $\Hom(X,X_0)\neq 0$. By duality, there is a non-zero map $\omega(X)\to E[r]$, giving rise to a triangle ${}_0\omega(X) \to \omega(X) \to E[r]$, where now ${}_0\omega(X) \in \ind\XX$ is the unique object with $h({}_0\omega(X))=0$ and $\Hom({}_0\omega(X),\omega(X))\neq 0$. At this point, $h(X)=m+r-1$ implies $X_0 = {}_0\omega(X)$, see e.g.\ the Hom hammock computations of \cite[\S2]{BPP}. The octahedral axiom applied to the composition of $X \to X_0$ and $X_0 \to \omega(X)$ yields a triangle $E[1] \to Q_X \to E[r]$. Since $E$ is exceptional, the connecting morphism $E[r]\to E[2]$ has to vanish, hence $E[1] \to Q_X \to E[r]$ splits as claimed.
(For a systematic computation of cones such as $Q_X$ in $\Db(\Lambda(r,n,m))$, see \cite[Appendix]{Broomhead}.)

Hence we get $\DD_X = {}\orth Q_X = {}\orth E$. In particular, $E$ exceptional shows that both conditions \SOT\ and \SOD\ are met; note $\Db(\Lambda(r,n,m))$ has a Serre functor because $\Lambda(r,n,m)$ is of finite global dimension if and only if $n>r$ --- an assumption we have made at the outset. Exactly the same reasoning works for $Y\in\ind\YY$ of height $n-r-1$.

If $r=1$, we use that $\Lambda(1,n,m) = A_m \tack \Lambda(1,n,0)$ is obtained by tacking $\A{m}$ onto the circular quiver algebra $\Lambda(1,n,0)$. Now $\Db(\Lambda(1,n,0))$ has up to shift a unique 0-spherical object $X'$ and the subgroup $\sod{[1],\tau} \subset \Aut(\Db(\Lambda(1,n,m)))$ acts transitively on 0-spherelike objects. Hence the embedding
 $\Db(\Lambda(1,n,0)) \into \Db(\Lambda(1,n,m))$
maps $X' \to \tau^s X[t]$ for some $s,t\in\IZ$. Applying \hyref{Lemma}{lem:sphericals_under_equivalences} for the autoequivalence $\tau^s[t]$ and \hyref{Theorem}{thm:tacking-sphericals} show
\[ \Db(\Lambda(1,n,m))_X \cong \Db(\Lambda(1,n,m))_{\tau^s X[t]} \cong \Db(\kk\A{m-1}) \oplus \Db(\Lambda(1,n,0)) . \]

If $r>1$, then the spherical subcategory $\DD_X$ is not the derived category of a finite-dimensional $\kk$-algebra: $\DD_X\cong\Db(A)$ for a finite-dimensional $\kk$-algebra $A$ implies, by \hyref{Lemma}{lem:smooth_spherical_subcategories}, that $A$ has finite global dimension.
However, that contradicts \hyref{Lemma}{lem:negative-spheres} and the fact that $X\in\DD_X$ is $(1-r)$-spherical.

In order to compute $\DD_Y$, note that $\Lambda(r,n,m)$ is an $\A{n-r-1}$-insertion of $\Lambda(r,r+1,m)$ at the successor of the trivalent vertex. Up to shift, there is a unique $(1+r)$-spherical object $Y'$ in $\Db(\Lambda(r,r+1,m))$. The embedding $\Db(\Lambda(r,r+1,m))\into\Db(\Lambda(r,n,m))$ maps $Y'\mapsto \tau^sY[t]$ for integers $s, t$, since the subgroup $\sod{[1], \tau} \subset \Aut(\Db(\Lambda(r,n,m)))$ acts transitively on $(1+r)$-spherelike objects in $\Db(\Lambda(r,n,m))$. \hyref{Theorem}{thm:insertion-sphericals} shows that 
\[ \Db(\Lambda(r,n,m))_{\tau^sY[t]} \cong \Db(\kk\A{n-r-2}) \oplus \Db(\Lambda(r,r+1,m)) . \]
In particular, since $(\tau^s[t])(\Db(\Lambda(r,n,m))_Y) \cong \Db(\Lambda(r,n,m))_{\tau^sY[t]}$, as in the proof of \hyref{Lemma}{lem:sphericals_under_equivalences}, the claim about $\DD_Y$ follows.
\end{proof}

\begin{remark}
It is easy to compute the left orthogonal of the (shifted) asphericality $Q_X=E$, using the Hom hammock computations of \cite[\S2]{BPP}.
From
 $\Hom^\bullet(A,E) = \Hom^\bullet(E,\Naka A)^*
                   = \Hom^\bullet(\tau\inv E,A)^*$, 
we get ${}\orth E = \tau\inv E\orth$. By \cite[\S2]{BPP}, the right orthogonal of $\tau\inv E$ consists of the full $\YY$ component and strips of triangles of width $m+r-2$ at the bases of all $\XX$ components together with the lattice consisting of $\sod{X}$, i.e.\ extensions and shifts of $X$.

Regarding the spherelike object $Y$, the indecomposables of $\DD_Y$ look like those of $\DD_X$, with the roles of $\XX$ and $\YY$ reversed: $\XX\subset\DD_Y$ and the indecomposables of $\DD_Y$ at the bottom of $\YY^0$ make up triangles of width $n-r-2$.
\end{remark}

\begin{figure}
\tikzdiscreteX
\caption{ \label{fig:discrete_X0}
The $\XX^0$ part of the spherical subcategory of $X\in\Db(\Lambda(2,n,2))$.
}
\end{figure}

\begin{figure}
\tikzdiscreteZ
\caption{ \label{fig:discrete_Z0} 
The $\ZZ^0$ part of the spherical subcategory of $X\in\Db(\Lambda(2,n,2))$.}
\end{figure}

\begin{example} \label{ex:DDA_spherical_subcategory}
Let $r=m=2$ and $n>2$ arbitrary, and $X$ as before. Write $\DD = \Db(\Lambda(2,n,2))$ and note $[2]|_\XX=\tau^{-4}$ on objects. \hyref{Figure}{fig:discrete_X0} depicts $\DD_X\cap\XX^0$. The shaded regions are the various Hom hammocks; the darker region on the left hand is $\{A\in\XX^0 \mid \Hom(A,E) = \Hom(\tau\inv E[-1],A)^*\neq0\}$, and the one on the right is $\{A\in\XX^0 \mid \Hom(\tau\inv E,A)\neq0\}$.
The intersection $\DD_X\cap\XX^1$ is not shown; it looks similar to \hyref{Figure}{fig:discrete_X0} and contains the odd shifts of $X$.

The additive subcategory generated by unshaded objects of heights $3+4k$ with $k \in \IN$ in $\XX^0$ and $\XX^1$ is $\sod{X}$, the triangulated category generated by $X$.

\hyref{Figure}{fig:discrete_Z0} shows the intersection $\DD_X\cap\ZZ^0$. Again, the picture for $\DD_X\cap\ZZ^1$ looks similar. Finally, we have $\DD_X\cap\YY=\YY$. Altogether, we get a complete and explicit description of $\DD_X$ in this example.
\end{example}

\begin{theorem}
The spherelike poset of $\DD = \Db(\Lambda(r,n,m))$ is
\begin{equation*}
   \poset(\DD) = 
   \begin{cases}
      \{ \DD \}                               & \text{if } (r,n,m) = (1,2,0); \\
      \{ \DD > \DD_{X_1},\ldots,\DD_{X_{m+r}} \} & \text{if }  r = n-1 \text{ and } m+r > 1; \\
      \{ \DD > \DD_{Y_1},\ldots,\DD_{Y_{n-r}} \} & \text{if }  r=1, m=0, n\geq2; \\
      \{ \DD_{X_1},\ldots,\DD_{X_{m+r}},\DD_{Y_1},\ldots,\DD_{Y_{n-r}} \} & \text{else,}
   \end{cases}
\end{equation*}
where $X\in\XX$ and $Y\in\YY$ are fixed nilpotent spherelike objects and $X_i\coloneqq \tau^{i-1}X$ 
and $Y_i\coloneqq \tau^{i-1}Y$.
\end{theorem}

\begin{proof}
Fix nilpotent spherelike objects $X\in\XX$ and $Y\in\YY$. They are both spherical if and only if $m=0$, $r=1$ and $n=2$. In this case, the poset $\poset(\DD)$ is reduced to one element.

Now $Y$ is spherical if and only if $n=r+1$. Given this, $X$ is properly spherelike if and only if $m>0$ or $r>1$; this is equivalent to $m+r>1$.
The spherical subcategory of $Y$ is $\DD=\DD_Y$ and $X, \tau X, \ldots, \tau^{m+r-1}X$ are $m+r$ pairwise incomparable properly spherelike objects, hence $\DD_Y > \DD_{\tau^i X}$.

Next, $X$ spherical and $Y$ properly spherelike is equivalent to $r=1, m=0$ and $n>1$. Here, the top element is $\DD = \DD_X$ and the botton elements are the spherical subcategories of $Y_1,\ldots,Y_{n-r}$ which differ by consecutive $\tau$-translations, hence are incomparable.

Finally, we have the `generic' case, when $\DD$ has no spherical objects. Here, we have the mutually incomparable spherelike objects $X_1,\ldots,X_{m+r}$ and $Y_1,\ldots,Y_{n-r}$, i.e.\ $\poset(\DD)$ is an antichain of cardinality $m+n$.
\end{proof}

The following corollary also holds for $r=n$, because then the $\XX$ components remain the same (and there are no $\YY$ components).

\begin{corollary} The $d$-spherelike poset of $\DD = \Db(\Lambda(r,n,m))$ has cardinality
\[  \#\posett_d(\DD) =
    \begin{cases}
       m+r, & \text{if } d=1-r; \\
       n-r, & \text{if } d=1+r; \\
       0    & \text{otherwise}.
    \end{cases}
\]
In particular, the parameters $r,n,m$ are determined by the integer-indexed sequence
 $(\#\posett_d(\Lambda(r,n,m)))_{d\in\IZ}$,
making this sequence a complete derived invariant for derived-discrete algebras.
\end{corollary}

\begin{corollary}
There are three possibilities for the stable spherelike poset $\posetaut(\DD)$ of $\DD=\Db(\Lambda(r,n,m))$:

\noindent
\begin{tabular}{ll}
$\posetaut = \{\bullet\}$           & if all spherelike objects are spherical, $(r,n,m)=(1,2,0)$; \\
$\posetaut = \{\bullet~~ \bullet\}$ & if there are no spherical objects; \\
$\posetaut = \{\bullet < \bullet\}$ & if spherical and properly spherelike objects exist.
\end{tabular}
\end{corollary}

\begin{proof}
As we look at the spherelike poset up to autoequivalences, all spherelike objects in $\XX$ become identified under AR-translation; same for $\YY$. If both $\XX$ and $\YY$ possess properly spherelike objects, then the associated spherical subcategories are incomparable by \hyref{Proposition}{prop:spherelikes_DDA}, leading to $\posetaut = \{\bullet~~ \bullet\}$. The other cases are obvious.
\end{proof}

\begin{corollary}
Let $\Lambda$ be an arbitrary finite-dimensional algebra and let $\DD \coloneqq \Db(\Lambda)$.
If $\posetaut(\DD)$ has cardinality greater than two, or if the height of $\poset(\DD)$ is greater than two, then $\Lambda$ is not a derived-discrete algebra.
\end{corollary}

\begin{remark}
Note that $\DD_X$ and $\DD_Y$ are the left orthogonal complements of certain exceptional objects; in particular, they form parts of a recollement of
$\Db(\Lambda)$ with their respective complements. There is an interesting contrast if we take some indecomposable object $Z\in\ZZ$ (which is always exceptional) instead: by \cite[Prop.~6.5]{BPP}, the orthogonal complement of $Z$ is the derived category of an iterated tilted algebra of type $A_{n+m-1}$; more precisely, there is a recollement $\recoll{ \Db(\kk A_{n+m-1})}{\Db(\Lambda(r,n,m))}{\sod{Z}}$. Note that $\Db(\kk\A{n})$ does not contain any nilpotent spherelike objects.
\end{remark}

\section{A non-commutative curve}
\label{sec:noncomm-curve}

\noindent
In \hyref{Example}{ex:non-commutative_curve}, we present a 3-spherelike object over a finite-dimensional algebra whose associated spherical subcategory is not of the form $\Db(\Lambda)$ for a finite-dimensional algebra $\Lambda$.

This object has a natural interpretation as a sheaf supported on the singularity  of a non-commutative nodal cubic curve. There is a dual 3-spherelike object with similarly peculiar spherical subcategory. However, the intersection of these two spherical subcategories has a clear interpretation, both algebraically and geometrically.

\begin{example} \label{ex:non-commutative_curve}
 We give an example of a $3$-spherelike object, such that the spherical subcategory is neither the derived category nor the perfect derived category of a finite-dimensional algebra. 
Consider the algebra $\Lambda = \Lambda_1$ from \hyref{Section}{sec:derived-invariant} given as
\[
\kk \Big( \xymatrix@C=1.3em{
          1 \ar@<0.5ex>[r]^a \ar@<-0.5ex>[r]_b &
          2 \ar@<0.5ex>[r]^a \ar@<-0.5ex>[r]_b & 3 }\Big) / (a^2,b^2)
\]
and the following $\Lambda$-module $E$
\[
\xymatrix@C=1.3em{
          \kk \ar@<0.5ex>[r]^0 \ar@<-0.5ex>[r]_1 &
          \kk \ar@<0.5ex>[r]^1 \ar@<-0.5ex>[r]_0 & \kk .}
\]
Then $E$ is an exceptional object of $\Db(\Lambda)$ and, moreover, $(\tau^{-1}E, E)$ forms an exceptional $(1,3)$-cycle in the sense of \cite{BPP}, i.e.\ $\nu E \cong (\tau^{-1}E)[3]$ (see e.g.\ \cite[Lem.~5.4]{BK}) and $\nu(\tau^{-1}E) \cong E[1]$ by definition of the AR-translation $\tau=\nu \circ [-1]$. There is an Auslander--Reiten triangle
 $E \to F \to \tau^{-1}E$ and by \cite[Prop.~3.7]{BPP}, $F$ is $3$-spherelike. Its asphericality can be computed to be $Q_F=E[1] \oplus E[-2]$.
In particular, $\Db(\Lambda)_F = {}\orth E$.

We use a K-theoretic argument first observed by Bondal, see e.g.~\cite{Kuznetsov}. Computing the Euler form $\chi(\blank,\blank)$ on $K(\Lambda)$ in the basis $S(1),S(2),S(3)$, and its restriction to the subgroup ${}\orth[E]$ gives
\[ \chi = \begin{pmatrix*}[r] 1 & -2 & 2 \\ 0 & 1 & -2 \\ 0 & 0 & 1 \end{pmatrix*} 
   \quad \text{and} \quad
   \chi|_{{}\orth[E]} = \begin{pmatrix*}[r] 0 & 1 \\ -1 & 0 \end{pmatrix*} ,
\]
where ${}\orth[E] = \genby{ [S(1)]+[S(2)], [S(2)]+[S(3)] }$. Therefore ${}\orth E$ has an anti-symmetric Euler form and consequently cannot be of the form $\Kb(\projectives A)$ or $\Db(A)$ for a finite-dimensional algebra $A$, since $\chi([A], [A])= \dim_\kk A \neq 0$. 
 
As $E$ is exceptional and $\Lambda$ has finite global dimension, condition \SOD\ and therefore also condition \SOT\ are satisfied. 

\addtocontents{toc}{\protect\setcounter{tocdepth}{1}}   
\subsection{Interpretation in non-commutative geometry}

Burban \& Drozd \cite[\S7]{BD} have shown that $\Lambda$ is derived equivalent to a non-commutative nodal cubic curve $\mathbb{X} \coloneqq (X, \mathcal{A})$, where $X \subset \mathbb{P}^2$ is a nodal cubic curve and $\AA=\EE nd_X(\OO_X\oplus\JJ)$ is a sheaf of $\OO_X$-algebras; here $\JJ\subset\OO_X$ is the ideal sheaf of the singularity. This equivalence $\mathbb{T}^{-1}\colon \Db(\Lambda) \to \Db(\mathbb{X})$ sends the $\Lambda$-module $E$ to the exceptional simple $\AA$-module $\SS_\gamma$ supported on the singular point of $X$, see \cite[Prop.~12]{BD}. 

Since the shift $[1]$ and the Nakayama (Serre) functor $\nu$ are equivalences, $F' \coloneqq \nu F [-2]$ is $3$-spherelike with asphericality $Q_{F'}=\nu E[-1] \oplus \nu E[-4] \eqqcolon E'[1] \oplus E'[-2]$. One can compute that $E'$ is the following exceptional $\Lambda$-module
\[
\xymatrix@C=1.3em{
          \kk \ar@<0.5ex>[r]^1 \ar@<-0.5ex>[r]_0 &
          \kk \ar@<0.5ex>[r]^0 \ar@<-0.5ex>[r]_1 & \kk }
\]
which corresponds to the other exceptional simple $\AA$-module $\SS_\alpha$ supported on the singular point of $X$, see \cite[Prop.~12]{BD}.

We claim that the intersection $\Db(\Lambda)_F \cap \Db(\Lambda)_{F'} \subseteq \Db(\Lambda)$ of the corresponding spherical subcategories is the full subcategory $\II$ of objects which are 
invariant under the Auslander--Reiten translation $\tau$. By \cite[Cor.~6]{BD} this subcategory corresponds to the category of perfect complexes $\Perf(X)$ on $X$ sitting inside the derived category of $\mathbb{X}$.

Since $E$ and $E'$ are $\frac{4}{2}$-fractionally Calabi--Yau (see e.g.\ \cite[Lem.~5.4]{BK}) and objects in $\II$ are $\smash{\frac{1}{1}}$-fractionally Calabi--Yau by definition, \cite[Lem.~5.3]{BK} provides the inclusion
$\II \subseteq {}\orth E \cap {}\orth E' = \Db(\Lambda)_F \cap \Db(\Lambda)_{F'}$. To see the other inclusion, we consider the image $q(T)$ of $T \in {}\orth E \cap {}\orth E'$ under the canonical projection
$q \colon \Db(\Lambda) \to \Db(\Lambda)/\II=:\mathcal T$. It suffices to show $q(T) \cong 0$. Firstly, for all $i \in \mathbb{Z}$
\begin{align*}
\Hom_\mathcal T(q(T), q(E[i] \oplus E'[i])) \cong \Hom_{\Db(\Lambda)}(T, E[i] \oplus E'[i]) =0 
\end{align*}
using \cite[proof of Cor.~5.5]{BK} in the first step and the definition of $T$ in the second. Applying \cite[Thm.~5.6]{BK} in combination with \cite[Thm.~4.8(a)]{BK}, the indecomposable objects in the idempotent completion $\mathcal T^\omega$ of $\mathcal T$ are completely understood. In combination with \cite[proof of Cor.~5.5]{BK}, this can be used to show that ${}\orth q(E \oplus E') = 0$. Hence $q(T) \cong 0$ which completes the argument.

To sum up, we have the following picture

\medskip
\noindent
\begin{tikzpicture}[description/.style={fill=white,inner sep=2pt}]
    \matrix (n) [matrix of math nodes, row sep=1em,
                 column sep=0.8em, text height=1.5ex, text depth=0.25ex,
                 inner sep=0pt, nodes={inner xsep=0.3333em, inner
ysep=0.3333em}]
    {  
      &             \Db(\Lambda)                         &&&&&& \Db(\mathbb{X}) \\ \\
      \Db(\Lambda)_F && \Db(\Lambda)_{F'} &&&&  {}\orth \mathcal S_\gamma && {}\orth \mathcal S_\alpha  \\ \\
     & \II &&&&&& \Perf(X) \\
     & \Db(\Lambda)_F \cap \Db(\Lambda)_{F'} &&&&&& {}\orth \mathcal S_\gamma \cap {}\orth \mathcal S_\alpha  \\  
    };

\path[right hook-> ] (n-3-1) edge (n-1-2);
\path[left hook-> ] (n-3-3) edge (n-1-2);

\path[left hook-> ] (n-5-2) edge (n-3-1);
\path[right hook-> ] (n-5-2) edge (n-3-3);

\path[-> ] (n-3-1) edge [bend left=10] node[ scale=0.75, fill=white] [midway] {$\nu \circ [-2]$} 
(n-3-3);
\path[-> ] (n-3-3) edge [bend left=10] node[ scale=0.75, fill=white] [midway] {$\nu \circ [-2]$} 
(n-3-1);

\draw[-] ($(n-5-2) + (-1pt, -7pt)$) edge
($(n-6-2)+ (-1pt, 7pt)$);
\draw[-] ($(n-5-2) + (1pt, -7pt)$) edge
($(n-6-2)+ (1pt, 7pt)$);



\path[right hook-> ] (n-3-7) edge (n-1-8);
\path[left hook-> ] (n-3-9) edge (n-1-8);

\path[left hook-> ] (n-5-8) edge (n-3-7);
\path[right hook-> ] (n-5-8) edge (n-3-9);

\path[-> ] (n-3-7) edge [bend left=10] node[ scale=0.75, fill=white] [midway] {$\nu \circ [-2]$} 
(n-3-9);
\path[-> ] (n-3-9) edge [bend left=10] node[ scale=0.75, fill=white] [midway] {$\nu \circ [-2]$} 
(n-3-7);

\draw[-] ($(n-5-8) + (-1pt, -7pt)$) edge
($(n-6-8)+ (-1pt, 7pt)$);
\draw[-] ($(n-5-8) + (1pt, -7pt)$) edge
($(n-6-8)+ (1pt, 7pt)$);


\path[-> ] ($(n-3-3.north)+(0, 18pt)$) edge [bend left=25] node[ scale=0.75, yshift=8pt] [midway] {$\mathbb{T}^{-1}$} node[ scale=0.75, yshift=-6pt] [midway] {$\sim$} 
($(n-3-7.north)+(0, 18pt)$);
\end{tikzpicture}
\end{example}

\begin{remark} \label{rem:Fukaya}
Versions of the above algebra $\Lambda$ come up in symplectic geometry: 
the very same algebra occurs in \cite[\S2C]{Seidel} as the double cover over a disk branched along 5 points ($g=1$). There, the algebra $\Lambda$ turns up with exactly the same relations $a^2=0, b^2=0$ as in \hyref{Example}{ex:non-commutative_curve}.

Moreover, by \cite[\S3A]{Seidel}, a version of it describes the directed Fukaya category of a Lefschetz pencil associated to $\OO(2)$ on $\IP^2$. In this reference, the relations are given by $a^2=b^2$ and $ab=ba$. For $\chara(k)\neq2$, the base change ($a\mapsto a-b$, $b\mapsto a+b$) presents the relation in the form of \hyref{Example}{ex:non-commutative_curve}.

In fact, all categories and algebras of \cite{Seidel} use $\IZ/2\IZ$ coefficients. Hence, the two algebras are not isomorphic. It would be interesting to see if the results of this section apply to the directed Fukaya category of \cite[\S3A]{Seidel}.

Finally, we mention that this is not the only appearance of gentle algebras in the context of Fukaya categories, see for example \cite{Haiden-etal} and \cite{Lekili-Polishchuk}.
\end{remark}

\begin{remark} \label{rem:cluster-tilting-connection}
We explain how the example above is related to the construction from cluster-tilting theory discussed in \hyref{Section}{sub:general}. The required definitions and properties of categories of maximal Cohen-Macaulay modules over isolated Gorenstein singularities can be found in \cite[\S1]{BIKR}.

Write $s \in X$ for the singular point and $\mathfrak{m}_s \subseteq \OO_s$ for the maximal ideal. We start with the following two rings, defined as completions of the (non-commutative and commutative) local rings at $s$: 
\[ A \coloneqq \varprojlim \AA_s/\mathfrak{m}_s^t \AA_s \qquad\text{and}\qquad
   R \coloneqq \widehat{\OO}_s = \varprojlim \OO_s/\mathfrak{m}_s^t \OO_s \cong \kk\llbracket x,y \rrbracket/(xy) \:, \]
By the proof of \cite[Thm.~2.6(2)]{BD},
\[ A = \End_R(R \oplus \mathfrak{m}_s) = \End_{R}(R \oplus \kk\llbracket x \rrbracket \oplus \kk\llbracket y \rrbracket) \]
is the Auslander algebra of the category of maximal Cohen-Macaulay $R$-modules $\mathsf{MCM}(R)$. This is a Frobenius category since $R$ is Gorenstein. Because $R$ has an isolated singularity, the stable category $\underline{\mathsf{MCM}}(R)$ is Hom-finite. Moreover, it is $0$-CY as $R$ is a curve singularity. Since $R$ is a hypersurface, $[2] \cong \mathsf{id}$ in $\underline{\mathsf{MCM}}(R)$ and thus this category is also $2$-CY. It has two $2$-cluster-tilting objects $\kk\llbracket x \rrbracket$ and $\kk\llbracket y \rrbracket$ by \cite[Prop.~2.4]{BIKR}. 

Let $B\coloneqq\End_R(R \oplus \kk\llbracket x \rrbracket)$ be the higher Auslander algebra. Then the cluster-tilted algebra
 $\underline{B}=\underline{\End}_R(R \oplus \kk\llbracket x \rrbracket)$ 
is isomorphic to $\kk$, see \cite[Prop.~2.4]{BIKR}. Therefore \hyref{Lemma}{lem:Keller-Reiten} applies to the simple $\underline{B}$-module $S$. This gives a $(2+1)$-spherical object $S \in \thick(\modules{\underline{B}}) \subseteq \Kb(\projectives{B})$. The triangulated category $\thick(\modules{\underline{B}})$ is generated by the spherical object $S$
--- Keller, Yang \& Zhou \cite{Keller-Yang-Zhou} describe
categories generated by $d$-spherical objects and show that they do not depend on the ambient triangulated category.

We collect the various categories and functors in a commutative diagram:

\medskip
\noindent
\begin{tikzpicture}[description/.style={fill=white,inner sep=2pt}]
    \matrix (n) [matrix of math nodes, row sep=0.7em,
                 column sep=0.8em, text height=1.5ex, text depth=0.25ex,
                 inner sep=0pt, nodes={inner xsep=0.3333em, inner
ysep=0.3333em}]
    {  
      \Kb(\projectives{B}) && \Kb(\projectives{A})  \\
      &&                                 \Db(A)                     && \Db(\mathbb{X}) && \Db(\Lambda)\\
      &&                                 \Db_{\rm fd}(A)        && \Db_{\{s\}}(\mathbb{X})\\
      \thick(\modules{\underline{B}}) && \thick(\modules{\underline{A}}) \\ \\
      \thick(S) && \thick(\widehat{(\SS_\alpha)}_s, \widehat{(\SS_\gamma)}_s) && \thick(\SS_\alpha, \SS_\gamma) && \thick(E', E) \\
      S && C
      && && F[1]
      \\  
    };

\path[right hook-> ] (n-1-1) edge node[ scale=0.75, yshift=8pt] [midway] {$Ae \otimes_B -$} (n-1-3);

\path[right hook-> ] (n-4-1) edge node[ scale=0.75, yshift=8pt] [midway] {$\mathbb{I}$} (n-4-3);

\path[right hook-> ] (n-6-1) edge node[ scale=0.75, yshift=8pt] [midway] {$\mathbb{I}$} (n-6-3);
\path[<- ] (n-6-3) edge node[ scale=0.75, yshift=5pt] [midway] {$\sim$} (n-6-5);
\path[<- ] (n-3-3) edge node[ scale=0.75, yshift=-5pt] [midway] {$\sim$} node[ scale=0.75, yshift=11pt] [midway] {$\widehat{(-)}_s$}(n-3-5);

\path[-> ] (n-6-5) edge node[ scale=0.75, yshift=5pt] [midway] {$\sim$} (n-6-7);

\path[-> ] (n-2-5) edge node[ scale=0.75, yshift=-5pt] [midway] {$\sim$} node[ scale=0.75, yshift=11pt] [midway] {$\mathbb{T}$}(n-2-7);

\path[|-> ] (n-7-1) edge (n-7-3);
\path[|-> ] (n-7-3) edge (n-7-7);


\path[right hook-> ] (n-4-1) edge node[ scale=0.75, xshift=-8pt] [midway] {$\iota$} (n-1-1);

\draw[-] ($(n-4-1) + (-1pt, -7pt)$) edge
($(n-6-1)+ (-1pt, 7pt)$);
\draw[-] ($(n-4-1) + (1pt, -7pt)$) edge
($(n-6-1)+ (1pt, 7pt)$);

\draw[-] ($(n-4-3) + (-1pt, -7pt)$) edge
($(n-6-3)+ (-1pt, 10pt)$);
\draw[-] ($(n-4-3) + (1pt, -7pt)$) edge
($(n-6-3)+ (1pt, 10pt)$);

\path[<-] (n-2-3) edge node[ scale=0.75, xshift=+9pt] [midway] {$\sim$}  (n-1-3);

\path[right hook-> ] (n-3-3) edge  (n-2-3);
\path[right hook-> ] (n-4-3) edge  (n-3-3);

\path[right hook-> ] (n-6-5) edge  ($(n-3-5)+(0, -10pt)$);
\path[right hook-> ] (n-6-7) edge  (n-2-7);

\path[right hook-> ] (n-3-5) edge  (n-2-5);
\end{tikzpicture}

Here, the idempotent $e \in A$ corresponds to the projection $A \onto R \oplus \kk\llbracket x\rrbracket$, and $\iota$ is induced by the projection $B \onto \underline{B}$, and
 $\underline{A}\coloneqq \underline{\End}_R(R \oplus \kk\llbracket x \rrbracket \oplus \kk\llbracket y \rrbracket)$. 
One can check that the composition $(Ae \otimes_B \blank) \circ \iota$ has image in $\thick(\modules{\underline{A}})$. This defines the inclusion $\mathbb{I} \colon \thick(\modules{\underline{B}}) \to \thick(\modules{\underline{A}})$. 

The right hand side of the diagram was already explained in \hyref{Example}{ex:non-commutative_curve} above. We note that
\[ C=\mathsf{cone}(\widehat{(\SS_\alpha)}_s[-1] \to \widehat{(\SS_\gamma)}_s[1]) \qquad\text{and}\qquad
   F[1]=\mathsf{cone}(E'[-1] \to E[1]) . \]
\end{remark}

\section{A tilted surface with spherelike poset of infinite width}

\begin{example} \label{ex:infinite_width}
Let $\DD = \Db(\Lambda)$ with the tensor algebra $\Lambda = \kk \tilde A_1\otimes_\kk \kk \tilde A_1$ of the Kronecker quiver
 $\smash{\tilde A_1} = \xymatrix@R=2ex{ 1 \ar@<0.5ex>[r] \ar@<-0.5ex>[r] & 2}$.
We show that $\poset(\DD)$ has infinite width and, particularly, infinite cardinality. 
 
Consider the idempotent $e=e_1 \otimes e_1 + e_1 \otimes e_2 \in \Lambda$. One can check that
 $e\Lambda e \cong \kk \tilde A_1 = \kk(\xymatrix@R=2ex{\bullet \ar@<0.5ex>[r] \ar@<-0.5ex>[r] & \bullet})$
and that $\Lambda/e \cong \Lambda (1 - e)$ is a projective $\Lambda$-module. This shows that the conditions of \hyref{Proposition}{prop:recollement} are satisfied.
 
 Thus $\jincl=\Lambda e \otimes_{e \Lambda e} (\blank)ﾊ\colon \Db(e\Lambda e) \to \Db(\Lambda)$ is an inclusion which has a right adjoint. Therefore we may apply \hyref{Theorem}{thm:projection_functor} to obtain a weak recollement
\[ \wrecoll{\thick(\modules {\Lambda/e}) \cap {}\orth\jincl(F_x)}{\DD_{\jincl(F_x)}}{\thick(\Lambda e)}, \]
where for each $x \in \IP^1$, $F_x$ denotes the corresponding quasi-simple $e \Lambda e$-module --- in other words, the modules sitting at the bottom of homogeneous tubes. Hence these are $1$-spherical objects.

For $y \in \IP^1$, let $G_y$ be the corresponding quasi-simple module over
 $\Lambda/e \cong \kk \tilde A_1 \cong \kk(\xymatrix@R=2ex{ \bullet \ar@<0.5ex>[r] \ar@<-0.5ex>[r] & \bullet})$. A computation in the homotopy category $\Kb(\projectives \Lambda)$ shows that $G_y \in {}\orth\jincl(F_x)$ if and only if $x \neq y$.
The weak recollement translates this to $G_y \in \Db(\Lambda)_{\jincl(F_x)}$ if and only if $x \neq y$. Consequently, the spherical subcategories $\smash{\Db(\Lambda)_{\jincl(F_x)}}$ are pairwise incomparable in $\DD$.

We remark that all these spherical subcategories $\smash{\Db(\Lambda)_{\jincl(F_x)}}$ become one element in $\posetaut(\Lambda)$, since $\Aut(\IP^1) \subseteq \Aut(\Db(\Lambda/e))$ acts transitively on tubes.

The above example corresponds to $\IP^1\times\IP^1$ since $\Db(\kk\tilde A_1)\cong\Db(\IP^1)$, where $F_x$ corresponds to $\pi^* \OO_p$ with $\pi \colon \IP^1\times\IP^1 \to \IP^1$ the projection to the first factor and $p \in \IP^1$. Here, by \cite[Ex.~5.7]{HKP}
\[ \Db(\IP^1\times\IP^1)_{\pi^* \OO_p} = \sod{\pi^* \Db_{\IP^1\setminus p}(\IP^1) \otimes \OO(-1,0),\pi^* \Db(\IP^1)} . \]
Now, let $X$ be $\IP^1\times\IP^1$ blown-up in $5$ points in general position, so $X$ is a del Pezzo surface of degree $3$. 
Since its anti-canonical bundle is ample, $\Aut(\Db(X))$ is generated by the shift, tensoring with line bundles and pullbacks via automorphisms; see \cite[Prop.~4.17]{Huybrechts}.
But by \cite{Koitabashi}, $\Aut(X)$ is trivial, so only tensoring with line bundles and the shift remain.

If we consider the composition $\tilde\pi\colon X \to \IP^1\times\IP^1 \xrightarrow{\pi} \IP^1$ and a point $p$ such that $\tilde\pi\inv(p)$ does not contain an exceptional divisor, then the spherical subcategory of $\tilde\pi^*\OO_p$ is essentially the same as above (where we add terms of the form $\OO_E(-1)$ for each exceptional divisor $E$; see \cite[\S5.2]{HKP}).
Since shift and twisting with a line bundle do not connect two different spherical subcategories of this kind (for they do not change the support), there is a subposet of infinite width in $\posetautt_1(X)$.
The surface $X$ admits a tilting bundle by \cite{Hille-Perling}, and so this geometric example also gives rise to a representation theoretic example.
\end{example}

\addtocontents{toc}{\protect\setcounter{tocdepth}{-1}}
\section*{Acknowledgements}

\noindent
We are grateful to Gustavo Jasso and Yanki Lekili for discussion and comments. 
Moreover, we thank Matthew Pressland for suggesting how to generalise \hyref{Lemma}{lem:Keller-Reiten} to \hyref{Remark}{rem:no-loops}. Finally we want to thank the anonymous referees for valuable comments.
Martin Kalck is grateful for the support by DFG grant Bu-1886/2-1 and EPSRC grant EP/L017962/1.

\setstretch{0.85}

\bigskip
\noindent
\resizebox{\textwidth}{!}{\emph{Contact:} \texttt{andreas.hochenegger@unimi.it, m.kalck@ed.ac.uk, david.ploog@ovgu.de}}

\end{document}